\documentclass[reqno,english]{amsart}
\usepackage[utf8]{inputenc}
\usepackage{xcolor}
\usepackage{babel}
\usepackage{mathrsfs}
\usepackage{enumitem}
\usepackage{amstext}
\usepackage{amsthm}
\usepackage{amssymb}
\usepackage{esint}
\PassOptionsToPackage{normalem}{ulem}
\usepackage{ulem}
\usepackage[pdfusetitle,
 bookmarks=true,bookmarksnumbered=false,bookmarksopen=false,
 breaklinks=false,pdfborder={0 0 0},pdfborderstyle={},backref=false,colorlinks=true]
 {hyperref}
\hypersetup{
 pdfborderstyle=,pdfborderstyle={}}

\makeatletter

\providecolor{lyxadded}{rgb}{0,0,1}
\providecolor{lyxdeleted}{rgb}{1,0,0}
\DeclareRobustCommand{\mklyxadded}[1]{\textcolor{lyxadded}\bgroup#1\egroup}
\DeclareRobustCommand{\mklyxdeleted}[1]{\textcolor{lyxdeleted}\bgroup\mklyxsout{#1}\egroup}
\DeclareRobustCommand{\mklyxsout}[1]{\ifx\\#1\else\sout{#1}\fi}
\numberwithin{equation}{section}
\numberwithin{figure}{section}
\theoremstyle{plain}
\newtheorem{thm}{\protect\theoremname}
\theoremstyle{definition}
\newtheorem{defn}[thm]{\protect\definitionname}
\theoremstyle{plain}
\newtheorem{prop}[thm]{\protect\propositionname}
\theoremstyle{plain}
\newtheorem{cor}[thm]{\protect\corollaryname}
\theoremstyle{remark}
\newtheorem{rem}[thm]{\protect\remarkname}
\theoremstyle{definition}
\newtheorem{example}[thm]{\protect\examplename}
\theoremstyle{plain}
\newtheorem{lem}[thm]{\protect\lemmaname}

\usepackage{pdfsync}
\usepackage{graphics}
\usepackage{epstopdf}
\usepackage[all]{xy}
\usepackage{amscd}
\usepackage{enumitem}
\usepackage{mathtools}
\usepackage{bbold}
\usepackage[utf8]{inputenc}
\setlist[enumerate]{leftmargin=*,label=(\roman*),align=left}

\makeatletter
\@namedef{subjclassname@2020}{%
  \textup{2020} Mathematics Subject Classification}
\makeatother

\newcommand{\ra}{\longrightarrow}
\newcommand{\field}[1]{\mathbb{#1}}
\newcommand{\R}{\field{R}} % reals
\newcommand{\N}{\field{N}} % naturals
\newcommand{\CC}{\mathbb C} % complex

\newcommand{\eps}{\varepsilon} % for the sake of brevity only
\renewcommand{\phi}{\varphi}
\newcommand{\diff}[1]{\ifmmode\mathchoice{\hbox{\rm d}#1}  % displaystyle
 {\hbox{\rm d}#1}  % normal 
 {\scalebox{0.75}{$\hbox{\rm d}#1$}}  % scriptstyle 
 {\scalebox{0.35}{$\hbox{\rm d}#1$}}  % scriptscriptstyle
 \fi} % dt,dx,... for integrals
\newcommand{\abs}[2][\empty]{\ifx#1\empty\left|#2\right|%
\else#1\vert #2 #1\vert\fi}% optional arg=size
\newcommand{\Rtil}{\widetilde \R} % real Colombeau generalized number
\newcommand{\Ctil}{\widetilde \CC} % complex Colombeau generalized number
\newcommand{\sint}[1]{\langle#1\rangle} % strong internal set
\newcommand{\frontRise}[2]{\ifmmode\mathchoice{{\vphantom{#1}}^{\scalebox{0.6}{$#2$}}}  % displaystyle
 {{\vphantom{#1}}^{\scalebox{0.56}{$#2$}}}  % normal 
 {{\vphantom{#1}}^{\scalebox{0.47}{$#2$}}}  % scriptstyle 
 {{\vphantom{#1}}^{\scalebox{0.35}{$#2$}}}\fi} % scriptscriptstyle 
\newcommand{\RC}[1]{\frontRise{\R}{#1}\Rtil}
\newcommand{\frontRiseDown}[3]{\ifmmode\mathchoice{{\vphantom{#1}}^{\scalebox{0.6}{$#2$}}_{\scalebox{0.6}{$#3$}}}  % displaystyle
 {{\vphantom{#1}}^{\scalebox{0.56}{$#2$}}_{\scalebox{0.56}{$#3$}}}  % normal 
 {{\vphantom{#1}}^{\scalebox{0.47}{$#2$}}_{\scalebox{0.47}{$#3$}}}  % scriptstyle 
 {{\vphantom{#1}}^{\scalebox{0.35}{$#2$}}_{\scalebox{0.35}{$#3$}}}\fi} % scriptscriptstyle 

\newcommand{\rti}{\RC{\rho}}

\newcommand{\ctil}{\frontRise{\CC}{\rho}\Ctil}

\makeatother

\providecommand{\corollaryname}{Corollary}
\providecommand{\definitionname}{Definition}
\providecommand{\examplename}{Example}
\providecommand{\lemmaname}{Lemma}
\providecommand{\propositionname}{Proposition}
\providecommand{\remarkname}{Remark}
\providecommand{\theoremname}{Theorem}

\begin{document}
\title{The Hahn-Banach theorem in spaces of nonlinear generalized functions}
\author{Djamel eddine Kebiche \and Paolo Giordano}
\thanks{This research was funded in whole or in part by the Austrian Science
Fund (FWF) 10.55776/P33945 and 10.55776/P33538. For open access purposes,
the authors have applied a CC BY public copyright license to any author-accepted
manuscript version arising from this submission}
\address{\textsc{Faculty of Mathematics, University of Vienna, Austria, Oskar-Morgenstern-Platz
1, 1090 Wien, Austria}}
\email{\texttt{paolo.giordano@univie.ac.at, djameleddine.kebiche@univie.ac.at}}
\begin{abstract}
In this paper, we establish a suitable version of the Hahn-Banach
theorem within the framework of Colombeau spaces, a class of spaces
used to model generalized functions. Our approach addresses the case
where maps are defined $\eps$-wise, which simplifies the framework
and makes the extension of linear functionals more manageable. As
an application of our main result, we demonstrate the separation of
convex sets in Colombeau spaces.
\end{abstract}

\maketitle

\section{Introduction}

\textcolor{black}{In analogy with the classical Hahn-Banach extension
property, one could ask if for any $\rti$-linear map $g$ that satisfies
\begin{equation}
\forall x\in\mathcal{F}\subseteq\mathcal{G}:\,g(x)\leq p(x),\label{eq:hbt}
\end{equation}
where $\mathcal{G}$ is an }$\rti$-module, $\mathcal{F}$ is an\textcolor{black}{{}
}$\rti$-submodule and $p$ is an $\rti$-sublinear map, \textcolor{black}{there
exists an $\rti$-linear map $f$ that extends $g$ and satisfies
}\eqref{eq:hbt} for every $x\in\mathcal{G}$\textcolor{black}{. Thm.~10.1
of \cite{Ver02} shows that under some set-theoretic assumptions }(continuum
hypothesis) there exists a linear $\rti$-functional $\phi$ defined
on a $\rti$-submodule of $\rti$ that cannot be extended to an $\rti$-linear
map $\psi:\rti\longrightarrow\rti$, and hence \textcolor{black}{the
extension property does not hold for any Colombeau space $\mathcal{G}_{E}$.
However, when $\mathcal{G}$ is a Hilbert $\rti$-module and if $\mathcal{F}$
is closed and edged (see Def.~2.9 of \cite{GaVe11}) the extension
property holds (see Thm~2.22 of \cite{GaVe11}). For a more restricted
version of the Hahn-Banach extension property on Banach $\rti$-modules,
see Thm.~5.2 of \cite{May07}, where the author proved that, when
$L\subseteq\rti$ is a subfield and $(E,||\cdot||)$ is an ultrametric
normed $L$-linear space, any continuous $L$-linear functional on
some $L$-linear subspace $V$ of $E$, can be extended to a continuous
$L$-linear functional on $E$. We also mention that the Hahn-Banach
extension property holds for non-Archimedean normed linear spaces
over a field $\mathbb{K}$ (see \cite{Lux76,ToVe08}).}

In our approach, we deal with the case where the $\rti$-module $\mathcal{G}$
is a Colombeau spaces $\mathcal{G}_{E}$ based on a family of normed
spaces or on a family of locally convex topological vector spaces
(see Sec.~\ref{subsec:ColSpa}), the maps $g$ and $p$ have representatives
(Def.~\ref{def:maps=000020with=000020rep=000020}), and the $\rti$-submodule
$\mathcal{F}$ is an internal set (Sec.~\ref{subsec:IntSet}). In
this setting, and $\eps$-characterization is possible allowing the
use of the classical Hahn-Banach theorem at the $\eps$-level and
conclude.

\section{Basic notions of Colombeau spaces }

This section introduces the foundational concepts of Colombeau spaces
necessary to understand the work presented in this paper. The main
references are \cite{Gar05,GaVe11,GKV24,LGKV,May07,MTAG21,ObVe08}.

\subsection{Colombeau spaces and topological $\ctil$-modules}

\subsubsection{\protect\label{subsec:ColSpa}Colombeau spaces}

Let $\rho:=(\rho_{\varepsilon})\in I^{I}$ ($I:=(0,1]$) be a net
such that $\rho_{\varepsilon}\to0$ as $\varepsilon\rightarrow0^{+}$
(in the following, such a net is called a \emph{gauge}) and \textcolor{black}{let
$E:=(E_{\varepsilon},(p_{\varepsilon n})_{n\in\mathbb{N}})_{\eps\in I}$
be a family of locally convex topological vector spaces such that
for every $\varepsilon$, the topology on $E_{\varepsilon}$ is defined
by a family of seminorms $(p_{\varepsilon n})_{n\in\mathbb{N}}$.
}In the following, all the asymptotic relations in $\eps\in(0,1]$
are considered for $\eps\to0^{+}$ 

\textcolor{black}{We set $\mathcal{G}_{E}:=\mathcal{M}_{E}\slash\mathcal{N}_{E}$,
where 
\begin{align*}
\mathcal{M}_{E} & :=\left\{ (u_{\varepsilon})\in\prod_{\eps\in I}E_{\varepsilon}\mid\forall n\in\N\,\exists N\in\mathbb{N}:\,p_{\varepsilon n}(u_{\varepsilon})=O(\rho_{\varepsilon}^{-N})\right\} ,\\
\mathcal{N}_{E} & :=\left\{ (u_{\varepsilon})\in\prod_{\eps\in I}E_{\varepsilon}\mid\forall n\in\N\,\forall N\in\mathbb{N}:\,p_{\varepsilon n}(u_{\varepsilon})=O(\rho_{\varepsilon}^{N})\right\} .
\end{align*}
}Elements of $\mathcal{M}_{E}$ and $\mathcal{N}_{E}$ are respectively
called $E$-moderate and $E$-negligible nets, whereas $\mathcal{G}_{E}$
is called the space of \textit{Colombeau generalized functions based
on $E=(E_{\varepsilon},(p_{\varepsilon n})_{n\in\mathbb{N}})$. }In
this work, equivalence classes in the quotient $\mathcal{G}_{E}$
are simply denoted as $[u_{\eps}]$.

For any $J\subseteq\N$, $E_{J}$ denotes the family $(E_{\varepsilon},(p_{\varepsilon n})_{n\in J})$
(and hence $E_{\N}=E$), and $\mathcal{G}_{E_{J}}$ denotes the space
of Colombeau generalized functions based on $E_{J}$. Note that any
element $u$ of $\mathcal{G}_{E}$ defines an element of $\mathcal{G}_{E_{J}}$
(which is still denoted by $u$).

\textcolor{black}{This type of $\rti$-module includes the usual Colombeau
algebras when we consider $\rho_{\eps}=\eps$, $E_{\varepsilon}=\mathcal{C}^{\infty}(\Omega)$,
where $\Omega\subseteq\R^{d}$ is an open set, and
\[
p_{\varepsilon n}(u)=\sup_{|\alpha|\leq n}\sup_{x\in K_{n}}|\partial^{\alpha}u(x)|
\]
where $(K_{n})_{n}$ is an exhaustive sequence of compact sets for
$\Omega$.}

\subsubsection{The Robinson-Colombeau ring of generalized numbers}

The rings $\rti$ and $\ctil$ are respectively the \textit{Robinson-Colombeau
ring of generalized real numbers }and the \textit{Robinson-Colombeau
ring of generalized complex numbers}, and they are obtained by taking
$E_{\varepsilon}=\mathbb{R}$ for every $\varepsilon$ and $E_{\varepsilon}=\mathbb{C}$
for every $\varepsilon$ respectively, see \cite{GKV24}. The generalized
number having $(\rho_{\eps})$ as a representative is denoted by $\diff\rho$.
Clearly, for $\rho_{\eps}=\eps$, we get the usual ring of Colombeau
generalized numbers, see e.g.~\cite{May07} and references therein.

The ring $\rti$ of generalized real numbers can be equipped with
the order relation $\leq$ given by $x\leq y$ if for any representative
$[x_{\eps}]$ of $x$, there exists a representative $[y_{\eps}]$
of $y$ such that $x_{\eps}\leq y_{\eps}$ for all $\eps$ small.
For abbreviation reason, the notation $\forall^{0}\eps$ refers to
``for all $\eps$ small''. Moreover, we write $x<y$ whenever $x\leq y$
and $y-x$ is a positive invertible generalized number. Equipped with
this order, $\rti$ is a partially ordered ring.

Given a subset $S\subseteq I$, we denote by $e_{S}$ the generalized
number having the characteristic function of $S$ as a representative.

\subsubsection{\protect\label{subsec:Top}Topological $\ctil$-modules}

Let $E$ be a family of locally convex topological vector spaces.
One can easily see that if $E_{\eps}$ is a vector space over $\R$
(resp.~over $\mathbb{C}$) for all $\eps$, then $\mathcal{G}_{E}$
is a module over $\rti$ (resp.~over $\ctil$). Thus, $\ctil$ (resp.~$\rti$)
is a module over itself.

We recall from Sec.~1.1.1 of \cite{GaVe11} that a \textit{topological
$\ctil$-module} is a $\ctil$-module $\mathcal{G}$ endowed with
a $\ctil$-linear topology, i.e.~, with a topology such that the
addition $\mathcal{G}\times\mathcal{G}\to\mathcal{G}:(u,v)\rightarrow u+v$
and the product $\ctil\times\mathcal{G}\to\mathcal{G}:\,(\lambda,u)\to\lambda u$
are continuous.

To endow $\ctil$ with a structure of a topological\textit{ }$\ctil$-module
we consider the map $\mathrm{v}_{\ctil}:\ctil\ra(-\infty,\infty]$
called \textit{valuation} defined by 
\[
\forall r=[r_{\eps}]\in\ctil:\,\mathrm{\mathrm{v}_{\ctil}}(r):=\sup\{b\in\R\mid|r_{\eps}|=O(\rho_{\eps}^{b})\,\text{as }\eps\to0^{+}\}.
\]
One can easily see that the map $\mathrm{v}$ is well-defined i.e.~it
does not depend on the representative $(r_{\eps})$ of $r$. The map
$\mathrm{v}$ satisfies the following properties
\begin{enumerate}
\item $\mathrm{v}(r)=\infty$ if and only if $r=0$
\item $\mathrm{v}(rs)\geq\mathrm{v}(r)+\mathrm{v}(s)$
\item $\mathrm{v}(r+s)\geq\min\{\mathrm{v}(r),\mathrm{v}(s)\}$.
\end{enumerate}
We also consider the map
\[
|\cdot|_{\mathrm{e}}:\ctil\ra[0,\infty),\,\,\,\,u\ra|u|_{\mathrm{e}}:=\mathrm{e}^{-\mathrm{\mathrm{v}_{\ctil}}(u)}.
\]
The properties of the valuation on $\ctil$ make the coarsest topology
on $\ctil$ such that the map $|\cdot|_{\mathrm{e}}$ is continuous
compatible with the ring structure. This topology is commonly called
the sharp topology \cite{NePi,Sca00,Sca98,Sca92}.

A \textit{locally convex topological $\ctil$-module }(\cite[Section 1.1.1]{GaVe11})\textit{
}is a topological $\ctil$-module whose topology is determined by
a family of \textit{ultra-pseudo-seminorm}s. As defined in \cite[Definition 1.8]{Gar05}
an \textit{ultra-pseudo-seminorm} on $\mathcal{G}$ is a map $\mathcal{P}:\,\mathcal{G}\ra[0,\infty)$
satisfying
\begin{enumerate}
\item $\mathcal{P}(0)=0$
\item $\mathcal{P}(\lambda u)\leq|\lambda|_{\mathrm{e}}\mathcal{P}(u)$
for all $\lambda\in\ctil$ and for all $u\in\mathcal{G}$
\item $\mathcal{P}(u+v)\leq\max\{\mathcal{P}(u),\mathcal{P}(v)\}$
\end{enumerate}
An ultra-pseudo-seminorm $\mathcal{P}$ satisfying $\mathcal{P}(u)=0$
if and only if $u=0$ is called \textit{ultra-pseudo-norm}. A deep
study of the topological dual of a topological $\ctil$-module $\mathcal{G}$
of all continuous and $\ctil$-linear functionals on $\mathcal{G}$
can be found in \cite{Gar05-1,Gar05}.

The notion of valuation can be defined in a general context of $\ctil$-modules
as follows: a valuation $\mathrm{v}$ on a $\ctil$-module $\mathcal{G}$
is a function $\mathrm{v}:\,\mathcal{G}\ra(-\infty,\infty]$ satisfying
\begin{enumerate}
\item $\mathrm{v}(0)=\infty$
\item $\mathrm{v}(\lambda u)\geq\mathrm{v}_{\ctil}(\lambda)+\mathrm{v}(u)$
for all $\lambda\in\ctil$ and for all $u\in\mathcal{G}$
\item $\mathrm{v}(u+v)\geq\min\{\mathrm{v}(r),\mathrm{v}(s)\}.$
\end{enumerate}
Any valuation $\mathrm{v}$ generates an ultra-pseudo-seminorm $\mathcal{P}$
by setting $\mathcal{P}(u)=\mathrm{e}^{-\mathrm{v}(u)}$.

In the context of a Colombeau space based on a family of locally convex
topological vector spaces\textcolor{black}{{} $E:=(E_{\varepsilon},(p_{\varepsilon n})_{n\in\mathbb{N}})_{\eps\in I}$},
for any $n$, the map $\mathrm{v}_{p_{n}}:\mathcal{G}_{E}\ra(-\infty,\infty]$
defined by 
\[
\mathrm{v}_{p_{n}}([u_{\eps}]):=\sup\{b\in\R\mid\,p_{\eps n}(u_{\eps})=O(\rho_{\eps}^{b})\,\text{as }\eps\ra0^{+}\}
\]
is a valuation on $\mathcal{G}_{E}$. To any valuation $\mathrm{v}_{p_{n}}$,
we associate the ultra-pseudo-seminorm $\mathcal{P}_{n}(u):=\mathrm{e}^{-\mathrm{v}_{p_{n}}(u)}$.
The family of ultra-pseudo-seminorms $\{\mathcal{P}_{n}\}$ equips
$\mathcal{G}_{E}$ with a structure of a locally convex topological
$\ctil$-module.

\textcolor{black}{A special kind of topological $\ctil$-modules is
$\ctil$-modules with $\rti$-seminorms ($\rti$-valued seminorms).
As defined in \cite[Definition 1.5]{GaVe11}, an $\rti$-seminorm
$p:\mathcal{G}\ra\rti$ is a map satisfying the following properties}
\begin{itemize}
\item $p(0)=0$ and $p(u)\geq0$ for all $u\in\mathcal{G}$;
\item $p(\lambda u)=|\lambda|p(u)$ for all $\lambda=[\lambda_{\eps}]\in\,\ctil$
and for all $u\in\mathcal{G}$, where $|\lambda|=[|\lambda_{\eps}|]\in\rti$;
\item $p(u+v)\leq p(u)+p(v)$ for all $u$, $v\in\mathcal{G}$.
\end{itemize}
If in addition, $p(u)=0$ implies $u=0$, we say that $p$ is an $\rti$-norm.

In the context of a Colombeau space based on a family of locally convex
topological vector spaces\textcolor{black}{{} $E:=(E_{\varepsilon},(p_{\varepsilon n})_{n\in\mathbb{N}})_{\eps\in I}$},
the map $\mathcal{G}_{E}\ra\rti$, $u\ra p_{n}(u):=[p_{\eps n}(u_{\eps})]$
defines an $\rti$-seminorm. In case $p_{\eps n}$ is a norm for all
$\eps$ small, the map $p_{n}$ is an $\rti$-norm. By \textcolor{black}{\cite[Proposition 1.6]{GaVe11},
for any $n\in\N$, the map $\mathcal{P}_{n}(u):=|p_{n}(u)|_{\mathrm{e}}$
defines an ultra-pseudo-seminorm on $\mathcal{G}_{E}$, and the $\ctil$-linear
topology on $\mathcal{G}_{E}$ determined by the family of $\rti$-seminorms
$\{p_{n}\}_{n\in\N}$ coincides with the topology of the corresponding
ultra-pseudo-seminorms $\{\mathcal{P}_{n}\}_{n\in\N}$.}

\bigskip{}

Throughout this paper, we consider the topology generated by $\rti$-seminorms
$(p_{n})$ in the case of a Colombeau space based on a family of locally
convex topological vector spaces, and by an $\rti$-norm in case of
a Colombeau space based on a family of normed spaces.

\subsubsection{\protect\label{subsec:IntSet}Internal sets and strongly internal
sets}

\textcolor{black}{We recall from \cite{ObVe08} that a subset $A\in\mathcal{G}_{E}$
is called internal if there exists a net $(A_{\eps})$ of subsets
$A_{\eps}\subseteq E_{\eps}$ such that 
\[
A:=\{u\in\mathcal{G}_{E}\mid\exists[u_{\eps}]=u\,\forall^{0}\eps:\,u_{\eps}\in A_{\eps}\}
\]
where the notation $\exists[u_{\eps}]=u$ refers to ``there exists
a representative $(u_{\eps})$ of $u$''. The net $(A_{\eps})$ is
said to be a representative of $A$, and the internal set $A$ is
said to be generated by the net $(A_{\eps})$. For a subset $J\subseteq\N$,
we denote by $A_{J}$ the internal set of $\mathcal{G}_{E_{J}}$ generated
by the net $(A_{\eps})$.}

An internal set $A$ is said to be sharply bounded if 
\[
\forall i\in\N\,\exists N_{i}\,\forall u\in A:\,p_{i}(u)\leq\diff\rho^{-N_{i}}.
\]
A representative $(A_{\eps})$ of an internal set is said to be sharply
bounded if 
\[
\forall i\in\N\,\exists N_{i}\,\forall^{0}\eps\,\forall u\in A_{\eps}:\,p_{\eps i}(u)\leq\rho_{\eps}^{-N_{i}}.
\]
Clearly, any internal set having a sharply bounded representative
is sharply bounded. The converse holds when $E$ is a family of normed
spaces \textcolor{black}{\cite[Lemma 3.3]{ObVe08}.}

When $E$ is a family of normed spaces, an internal set $A$ is said
to be functionally compact set (\cite[Definition 52]{GKV24}), and
we write $A\Subset_{\mathrm{f}}\mathcal{G}_{E}$, if $A$ has a sharply
bounded representative made of compact subsets. This type of sets
are not necessary compact sets with respect to the sharp topology,
but they satisfy many properties similar to those satisfied by compact
sets in normed spaces.

Any internal set $A\subseteq\mathcal{G}_{E}$ is closed (with respect
to the sharp topology). For more results about internal sets, we refer
the reader to \textcolor{black}{\cite{ObVe08}.}

A subset $A\subseteq\mathcal{G}_{E}$ is said to be strongly internal
if there exists a net $(A_{\eps})$ of subsets $A_{\eps}\subseteq E_{\eps}$
such that \textcolor{black}{
\[
A:=\{u\in\mathcal{G}_{E}\mid\forall[u_{\eps}]=u\,\forall^{0}\eps:\,u_{\eps}\in A_{\eps}\}.
\]
where the notation $\forall[u_{\eps}]=u$ refers to ``for any representative
$(u_{\eps})$ of $u$''.}

\subsubsection{\protect\label{subsec:Sup}Supremum in $\rti$}

We recall here the notion of (sharp) supremum of subsets of $\rti$
(see Def.~20 and Rem.~21 of \cite{MTAG21} and references therein),
which are defined with respect to the sharp topology which is, as
defined in Sec.~\ref{subsec:Top}, generated by the family 
\[
\{B_{r}(x)\mid r\in\rti_{>0}\,\,\,x\in\rti\}
\]
where $B_{r}(x):=\{y\in\rti\mid|x-y|<r\}$ is the open ball of $\rti$
with center $x$ and radius $r$.
\begin{defn}
Let $S\subseteq\rti$, and let $\sigma$, $\iota\in\rti$. Then we
say that $\sigma$ is the \emph{sharp }(or \emph{closed})\emph{ supremum}
of $S$, and we write $\sigma=\sup(S)$, if
\begin{enumerate}
\item $\forall s\in S:\,s\leq\sigma$
\item $\forall q\in\N\,\exists\overline{s}\in S:\,\sigma-\diff\rho^{q}\leq\overline{s}$
\end{enumerate}
Similarly, we say that $\iota$ is the sharp infimum of $S$, and
we write $\iota=\inf(S)$ if
\begin{enumerate}
\item $\forall s\in S:\,s\geq\iota$
\item $\forall q\in\N\,\exists\overline{s}\in S:\,\iota+\diff\rho^{q}\geq\overline{s}$.
\end{enumerate}
Unlike in $\R$, there exist bounded subsets of $\rti$ that do not
have a supremum. Even worse, there exists subsets of $\rti$ that
do not have a least upper bound e.g.~the set $D_{\infty}$ of all
infinitesimal points i.e.~
\[
D_{\infty}:=\{x\in\rti\mid\exists[x_{\eps}]=x\,\lim_{\eps\to0^{+}}x_{\eps}=0\}=\{x\in\rti\mid x\approx0\}.
\]
For more details, see \cite{MTAG21}.
\end{defn}

\subsubsection{\protect\label{subsec:LinMaps}$\ctil$-(sub)linear maps and continuity}

Let $(\mathcal{G},\{p_{l}\}_{l\in L})$ and $(\mathcal{F},\{q_{j}\}_{j\in J})$
be topological $\ctil$-modules with $\rti$-seminorms. We recall
from Prop.~1.7 of \cite{GaVe11} that a $\ctil$-linear map $T:\mathcal{G}\ra\mathcal{F}$
is continuous if and only if for all $j\in J$ there exist $C\in\rti_{\geq0}$
and a finite subset $L_{0}\subseteq L$ such that 
\[
\forall u\in\mathcal{G}:\,q_{j}(T(u))\leq C\max_{l\in L_{0}}p_{l}(u).
\]
As a consequence, in the case of a Colombeau space based on a family\textit{$E=(E_{\varepsilon},(p_{\varepsilon n})_{n\in\mathbb{N}})$}
of locally convex topological vector spaces, a $\ctil$-linear map
$T:\mathcal{G}_{E}\ra\rti$ is continuous if there exist a constant
$C\in\rti_{\geq0}$ and a finite subset $J\subseteq\N$ such that
\[
\forall u\in\mathcal{G}_{E}:\,|T(u)|\leq C\max_{j\in J}p_{j}(u).
\]
Let $\mathcal{G}$ be an $\rti$-module. A functional $p:\mathcal{G}\longrightarrow\rti$
is said to be $\rti$-\emph{sublinear map} if it fulfills the following
two conditions 
\begin{align}
\forall u\in\mathcal{G\,}\forall\lambda\in\rti_{>0}:\,p(\lambda u) & =\lambda p(u)\label{eq:sublin1}\\
\forall u,v\in\mathcal{G}:\,p(u+v)\leq & p(u)+p(v).\label{eq:sublin2}
\end{align}
We recall now the definition of $\ctil$-(sub)linear maps defined
by a net of representatives.
\begin{defn}
\label{def:maps=000020with=000020rep=000020}Let $E=(E_{\eps},(p_{\eps n})_{n\in\N})$
be a family of locally convex topological vector spaces, and let $T:\mathcal{G}_{E}\ra\ctil$
be a $\ctil$-(sub)linear map. We say that $T$ is \textit{defined
by a net of representatives} if there exists a net $(T_{\eps})$ of
$\mathbb{C}$-(sub)linear maps $T_{\eps}:E_{\eps}\ra\mathbb{C}$ such
that 
\[
\forall(u_{\eps}),\,(v_{\eps})\in\mathcal{M}_{E}:\,(T_{\eps}(u_{\eps}))\in\mathcal{M}_{\mathbb{C}}\,\,\,\text{and}\,\,\,(u_{\eps}-v_{\eps})\in\mathcal{N}_{E}\Rightarrow\,(T_{\eps}(u_{\eps})-T_{\eps}(v_{\eps}))\in\mathcal{N}_{\mathbb{C}}
\]
and $T(u)=[T_{\eps}(u_{\eps})]$ for all $u=[u_{\eps}]\in\mathcal{G}_{E}$.
In the sequel, we write $T=[T_{\eps}]$ to say that $T$ is \textit{defined
by a net of representatives} $(T_{\eps})$.

Moreover, we say that the map $T$ is well-defined on $\mathcal{G}_{E_{J}}$
($J\subseteq\N$) if the latter holds with $E_{J}$ instead of $E$.
\end{defn}

It follows from the definition that $(T_{\eps})$ and $(T'_{\eps})$
are two representatives of a $\ctil$-(sub)linear map $T:\mathcal{G}_{E}\ra\ctil$
if and only if 
\[
\forall(u_{\eps})\in\mathcal{M}_{E}:\ (T_{\eps}(u_{\eps})-T'_{\eps}(u_{\eps}))\in\mathcal{N}_{\mathbb{C}}.
\]

We also recall from Sec.~1.1.2 of \cite{GaVe11} that a map $T:\mathcal{G}_{E}\ra\mathbb{C}$
is called basic if it has a representative $(T_{\eps})$ of continuous
linear maps $E_{\eps}\ra\mathbb{C}$ that satisfies the following
property 
\[
\exists J\subseteq\N\text{ finite}\,\exists N\in\N\,\forall^{0}\eps\,\forall u\in E_{\eps}:\,|T_{\eps}(u)|\leq\rho_{\eps}^{-N}\sum_{j\in J}p_{\eps j}(u).
\]
It is clear that any basic map is continuous, and in the following
proposition we prove that an arbitrary $\ctil$-linear map $T:\mathcal{G}_{E_{J}}\ra\ctil$
defined componentwise is continuous when $J$ is finite.
\begin{prop}
\label{prop:=000020representative=000020implies=000020continuous=000020}
Let $E=(E_{\eps},(p_{\eps n})_{n\in\N})$ be a family of locally convex
topological vector spaces, and let $T:\mathcal{G}_{E}\ra\ctil$ be
a $\ctil$-linear map with representatives. If there exists a finite
subset $J\subseteq\N$ such that $T$ is well defined on $\mathcal{G}_{E_{J}}$,
then $T$ is continuous on $\mathcal{G}_{E_{J}}$ and hence on $\mathcal{G}_{E}$.
\end{prop}

\noindent The proposition generalizes \cite[Prop.~1.9]{GaVe11} where
$E$ is a family of normed spaces, so that the proof is very similar.
\begin{proof}
By assumption, the net $(T_{\eps})$ satisfies 
\begin{equation}
\forall(u_{\eps})\in\mathcal{N}_{E_{J}}:\,(T_{\eps}(u_{\eps}))\in\mathcal{N}_{\mathbb{C}}\label{eq:1.22-2}
\end{equation}
We claim that
\begin{equation}
\forall n\in\N\,\exists m\in\N\,\forall^{0}\eps\,\forall u\in E_{\eps}:\,\max_{j\in J}p_{\eps j}(u)\leq\rho_{\eps}^{m}\Longrightarrow|T_{\eps}(u)|\leq\rho_{\eps}^{n}.\label{eq:1.22-1}
\end{equation}
Indeed, if \eqref{eq:1.22-1} does not hold, then there exist $n'\in\N$
and a non-increasing sequence $(\eps_{m})_{m}\downarrow0$ as $m\to+\infty$,
and a sequence $(u_{\eps_{m}})_{m}$ of $\prod_{m\in\N}E_{\eps_{m}}$
such that 
\[
\forall m\in\N:\,\max_{j\in J}p_{\eps_{m}j}(u_{\eps_{m}})\leq\rho_{\eps_{m}}^{m}\,\,\,\text{and}\,\,\,|T_{\eps_{m}}(u_{\eps_{m}})|>\rho_{\eps_{m}}^{n'}.
\]
Define $u_{\eps}=u_{\eps_{m}}$ when $\eps=\eps_{m}$ and $u_{\eps}=0$
otherwise. By construction, we have that $(u_{\eps})\in\mathcal{N}_{E_{J}}$
and $\forall m\in\N:\,|T_{\eps_{m}}(u_{\eps_{m}})|>\rho_{\eps_{m}}^{n'}$,
which contradicts \eqref{eq:1.22-2}. Assertion \eqref{eq:1.22-1}
says that for all $n\in\N$ there exists $m\in\N$ such that $|T(u)|\leq\diff\rho^{n}$
if $\max_{j\in J}p_{j}(u)\leq\diff\rho^{m}$. Thus, $T$ is continuous
at the origin and hence continuous on $\mathcal{G}_{E_{J}}$.
\end{proof}
\textcolor{black}{Consider now a continuous $\ctil$-linear map $T=[T_{\eps}]:\mathcal{G}_{E}\ra\ctil$
where $E=(E_{\eps},||\cdot||_{\eps})$ is a family of normed spaces.
Then 
\begin{equation}
\exists\sup\{|T(x)|\mid x\in\mathcal{G}_{E},\,||x||\leq1\}=:|||T|||\in\rti\label{eq:SupNorm}
\end{equation}
and it is given by 
\begin{equation}
|||T|||=\left[\sup_{x\in E_{\varepsilon}\,||x||_{\varepsilon}\leq1}|T_{\varepsilon}(x)|\right]=:\left[|||T_{\eps}|||\right].\label{eq:norm}
\end{equation}
Indeed, from Prop.~1.10 of \cite{GaVe11}, there exists a representative
$(C_{\eps})$ of $C$ such that }
\begin{equation}
\forall^{0}\eps\,\forall x\in E_{\eps}:\,|T_{\eps}(x)|\leq C_{\eps}||x||_{\eps}.\label{eq:1.25}
\end{equation}
Setting 
\begin{equation}
|||T_{\eps}|||_{\eps}:=\sup_{x\in E_{\eps},||x||_{\eps}\leq1}|T_{\eps}(x)|.\label{eq:1.25-1}
\end{equation}
It follows from \eqref{eq:1.25} that $[|||T_{\eps}|||_{\eps}]\in\rti$.
Setting now $|||T|||:=[|||T_{\eps}|||_{\eps}]$. One can easily check
that \eqref{eq:SupNorm} follows from \eqref{eq:1.25} and \eqref{eq:1.25-1}.

\textcolor{black}{The set $I$ used in all the constructions above
can be replaced by any subset$S\subseteq(0,1]$ having $0$ as an
accumulation point. Such subsets are called }\textit{co-final in $I$
}\textcolor{black}{and we write $S\subseteq_{0}I$. However, when
$I$ is replaced by a co-final subset $S$ of $I$, the ring $\ctil$
will be denoted by $\ctil|_{S}$, the $\ctil|_{S}$-module $\mathcal{G}_{E}$
will be denoted by $\mathcal{G}_{E}|_{S}$, the $\ctil|_{S}$-linear
map $T$ will be denoted by $T|_{S}$, the order relation $\leq$
of $\rti$ will be denoted by $\leq_{S}$, the equivalence classes
will be denoted by $[\cdot]|_{S}$; similar notations are used for
other notions.}

\section{\textcolor{black}{\protect\label{sec:The-analytic-form}The analytic
form of Hahn-Banach's theorem}}

In this paper, we prove the extension property in two different cases:
in a Colombeau space based on a family of normed spaces (Thm.~\ref{thm:HBT}),
and in a Colombeau space based on a family of locally convex topological
vector spaces (Thm.~\ref{thm:HBT=000020LCTVS}). In the two cases,
the method is basically the same but the techniques used are different.
However, in both cases we simplify the problem to the $\eps$-level
and then we use the classical Hahn-Banach theorem to conclude. Moreover,
the maps $g$ and $p$ in \eqref{eq:hbt} are assumed to be defined
componentwise, and the $\rti$-submodule $\mathcal{F}$ is assumed
to be an internal set.

In all this section, an internal $\rti$-submodule is assumed to be
generated by a representative consisting of vector spaces.

\subsection{The Hahn-Banach theorem in a Colombeau space based on a family of
normed spaces}

In all this section, we assume that $E:=(E_{\eps},||\cdot||_{\eps})$
is a family of normed spaces.
\begin{thm}
\label{thm:HBT}Let $\mathcal{F}\subseteq\mathcal{G}_{E}$ be an $\rti$-submodule,$p:\mathcal{G}_{E}\ra\rti$
be an $\rti$-sublinear map and let $g:\text{\ensuremath{\mathcal{F}\longrightarrow\rti}}$
be an $\rti$-linear map satisfying 
\begin{equation}
\forall x\in\mathcal{F}:\,g(x)\leq p(x).\label{eq:bound-condition}
\end{equation}
Assume that $\mathcal{F}$ is internal, and both $p$ and $g$ have
representatives. Then, there exists an $\rti$-linear map $f:\mathcal{G}_{E}\ra\rti$
having representatives and satisfying 
\begin{equation}
\forall x\in\mathcal{F}:\,f(x)=g(x)\label{eq:extension1}
\end{equation}
and 
\begin{equation}
\forall x\in\mathcal{G}_{E}:\,f(x)\leq p(x).\label{eq:extension2}
\end{equation}
\end{thm}

\noindent\textcolor{black}{The case where the map $g$ takes vales
in $\ctil$ can be treated in a very similar way.}

\textcolor{black}{The next proposition, which is inspired by Prop.
1.10 of \cite{GaVe11}, is used in order to reduce \eqref{eq:bound-condition}
at the $\eps$-level.}
\begin{prop}
\label{prop:reduce}Under the assumptions of Thm.~\ref{thm:HBT},
the following properties are equivalent:
\begin{enumerate}
\item \label{enu:reduce1}$\forall x\in\mathcal{F}:\,g(x)\leq p(x)$,
\item \label{enu:reduce2}For all representatives $(g_{\varepsilon})$ of
$g$, for all representatives $(p_{\varepsilon})$ of $p$, we have
\begin{equation}
\forall q\in\N\,\forall^{0}\eps\,\forall x\in F_{\varepsilon}:\,g_{\varepsilon}(x)\leq p_{\varepsilon}(x)+\rho_{\varepsilon}^{q}||x||_{\varepsilon}\label{eq:reduce2eq}
\end{equation}
where $(F_{\eps})$ is a representative of $\mathcal{F}$ made of
vector spaces,
\item \label{enu:reduce3}For all representatives $(g_{\varepsilon})$ of
$g$, there exists a representative $(p'_{\varepsilon})$ of $p$
such that 
\begin{equation}
\forall^{0}\varepsilon\,\forall x\in F_{\varepsilon}:\,g_{\varepsilon}(x)\leq p'_{\varepsilon}(x).\label{eq:reduce3eq}
\end{equation}
\end{enumerate}
\end{prop}

\begin{proof}
\ref{enu:reduce1}$\ \Rightarrow\ $\ref{enu:reduce2}: In order to
prove this implication, we prove that the negation of \ref{enu:reduce2}
implies the negation of \ref{enu:reduce1}. More precisely, we should
prove that the negation of \ref{enu:reduce2} implies that there exist
$S\subseteq_{0}(0,1]$ and $x\in\mathcal{F}$ such that $g_{\eps}(x_{\eps})>p_{\eps}(x_{\eps})+\rho_{\eps}^{q}$
for some $q\in\N$ and for all $\eps\in S$. From 
\[
\exists(g_{\varepsilon})\,\exists(p_{\varepsilon})\,\exists q\in\mathbb{N}\,\forall\eta\in(0,1]\,\exists\varepsilon\in(0,\eta]\,\exists x_{\varepsilon}\in F_{\varepsilon}:\,g_{\varepsilon}(x_{\eps})>p_{\varepsilon}(x_{\varepsilon})+\rho_{\varepsilon}^{q}||x_{\varepsilon}||_{\varepsilon}
\]
we have that there exist a non-increasing sequence $(\varepsilon_{k})_{k}\downarrow0$
and a sequence $x_{\varepsilon_{k}}\in F_{\varepsilon_{k}}$ such
that 
\[
\forall k\in\N:\,g_{\varepsilon_{k}}\left(x_{\varepsilon_{k}}\right)>p_{\varepsilon_{k}}\left(x_{\varepsilon_{k}}\right)+\rho_{\varepsilon_{k}}^{q}\left\Vert x_{\varepsilon_{k}}\right\Vert _{\varepsilon_{k}}.
\]
Set $S:=\{\eps_{k}\mid k\in\N\}\subseteq_{0}(0,1]$. Since $p_{\eps_{k}}$
is sublinear, and $g_{\eps_{k}}$ is linear, we have that $x_{\varepsilon_{k}}\neq0$
for every $k\in\N$ and hence we can assume that $\left\Vert x_{\varepsilon_{k}}\right\Vert _{\varepsilon_{k}}=1$
for every $k$. Without loss of generality, we can assume that $F_{\eps}\ne\{0\}$
for all $\eps$. Let $(u_{\varepsilon})\in\prod_{\eps\in I}F_{\varepsilon}$
with $||u_{\varepsilon}||_{\varepsilon}=1$ for every $\varepsilon$.
The net $x_{\varepsilon}=x_{\varepsilon_{k}}$ when $\varepsilon=\varepsilon_{k}$
and $x_{\varepsilon}=u_{\varepsilon}$ otherwise generates an element
$x=[x_{\varepsilon}]\in\mathcal{F}$ with $\rti$-norm $1$. By construction,
for all $\eps\in S$ we have that 
\[
g_{\eps}(x_{\eps})>p_{\varepsilon}(x_{\varepsilon})+\rho_{\varepsilon}^{q}
\]
which contradicts \ref{enu:reduce1}.

\ref{enu:reduce2}$\ \Rightarrow\ $\ref{enu:reduce3}: Let's fix
a representative $(g_{\varepsilon})$ of $g$ and a representative
$(p_{\varepsilon})$ of $p$. From \ref{enu:reduce2} we can extract
a non-increasing sequence $(\eta_{q})_{q}$ converging to $0$ such
that 
\begin{equation}
\forall\varepsilon\in(0,\eta_{q}]\,\forall x\in F_{\varepsilon}:\,g_{\varepsilon}(x)\leq p_{\varepsilon}(x)+\rho_{\varepsilon}^{q}||x||_{\varepsilon}.
\end{equation}
The net $(n_{\eps})$ defined by $n_{\varepsilon}=\rho_{\varepsilon}^{q}$
for $\varepsilon\in(\eta_{q+1},\eta_{q}]$ is negligible. Since every
norm is a sublinear map, the net $(p'_{\eps})$ defined by $p'_{\varepsilon}:=p_{\varepsilon}+n_{\varepsilon}||\cdot||_{\varepsilon}$
for all $\eps$ is a representative of $p$ and \eqref{eq:reduce3eq}
holds for $\eta=\eta_{1}$.

\ref{enu:reduce3}$\ \Rightarrow\ $\ref{enu:reduce1}: This one is
clear.
\end{proof}
\begin{proof}[Proof of Thm.~\ref{thm:HBT}]
 Let $(g_{\varepsilon})$ be a representative of $g$. By \ref{enu:reduce3}
of Prop.~\ref{prop:reduce}, there exists a representative $(p_{\varepsilon})$
of $p$ such that 
\[
\forall^{0}\varepsilon\,\forall x\in F_{\varepsilon}:\,g_{\varepsilon}(x)\leq p_{\varepsilon}(x).
\]
Applying the classical Hahn-Banach theorem to $g_{\eps}$ for every
$\eps$ small, we obtain a family $(f_{\varepsilon})$ of linear maps$f_{\eps}:E_{\eps}\ra\R$
that satisfies 
\begin{equation}
\forall^{0}\varepsilon\,\forall x\in F_{\varepsilon}:\,f_{\varepsilon}(x)=g_{\varepsilon}(x)\label{eq:2.9}
\end{equation}
and 
\begin{equation}
\forall^{0}\varepsilon\,\forall x\in E_{\varepsilon}:\,f_{\varepsilon}(x)\leq p_{\varepsilon}(x).\label{eq:extension}
\end{equation}
Clearly, by \eqref{eq:extension} we have 
\[
\forall(u_{\eps})\in\mathcal{M}_{E}:\,(f_{\eps}(u_{\eps}))\in\mathcal{M}_{\mathbb{C}}\,\,\,\text{and}\,\,\,\forall(u_{\eps})\in\mathcal{N}_{E}:\,(f_{\eps}(u_{\eps}))\in\mathcal{N}_{\mathbb{C}},
\]
so we can consider the $\rti$-linear map $f:=[f_{\eps}]$, which
by \eqref{eq:2.9} and \eqref{eq:extension}, satisfies \eqref{eq:extension1}
and \eqref{eq:extension2}.
\end{proof}
\begin{cor}
\textcolor{black}{\label{cor:extension=000020continuous}Let $\mathcal{F}\subseteq\mathcal{G}_{E}$
be an internal $\rti$-submodule and let $g:\text{\ensuremath{\mathcal{F}\longrightarrow\rti}}$
be a continuous $\rti$-linear map with representatives. Then, there
exists a continuous $\rti$-linear map $f:\mathcal{G}_{E}\longrightarrow\rti$
with representatives that extends $g$ and satisfies 
\begin{equation}
|||g|||=|||f|||.\label{eq:norm2}
\end{equation}
}
\end{cor}

\textcolor{black}{}
\begin{proof}
\textcolor{black}{Taking $p(x)=|||g|||\cdot||x||=[|||g_{\eps}|||\cdot||x_{\eps}||]$.
By Thm.~\ref{thm:HBT} there exists an $\rti$-linear map $f=[f_{\varepsilon}]:\mathcal{G}_{E}\longrightarrow\rti$
that extends $g$, and satisfies 
\begin{equation}
\forall x\in\mathcal{G}_{E}:\,f(x)\leq|||g|||\cdot||x||.\label{eq:2.11}
\end{equation}
From \eqref{eq:extension1} and from }\eqref{eq:norm}\textcolor{black}{{}
we have 
\begin{multline*}
|||g|||=\sup_{x\in\mathcal{F\thinspace}||x||\leq1}|g(x)|=\left[\sup_{x\in F_{\varepsilon}\thinspace||x||_{\varepsilon}\leq1}|g_{\varepsilon}(x)|\right]\\
=\left[\sup_{x\in F_{\varepsilon}\thinspace||x||_{\varepsilon}\leq1}|f_{\varepsilon}(x)|\right]\leq\left[\sup_{x\in E_{\varepsilon}\thinspace||x||_{\varepsilon}\leq1}|f_{\varepsilon}(x)|\right]=|||f|||.
\end{multline*}
On the other hand, the inequality $|||f|||\leq|||g|||$ follows from
}\eqref{eq:2.11}.
\end{proof}
\textcolor{black}{One of the classical applications of the Hahn-Banach
theorem is to prove that, for some given element $u$ of a normed
space $(E,||\cdot||)$, there exists $f\in E'$ (the topological dual
of $E$) such that 
\[
||f||_{E'}=||u||_{E}\text{ and }f(u)=||u||^{2}.
\]
The following theorem is a generalization of this result.}
\begin{cor}
\textcolor{black}{Let $u\in\mathcal{G}_{E}$. Then, there exists a
continuous $\rti$-linear map $f:\mathcal{G}_{E}\longrightarrow\rti$
with representatives such that 
\[
|||f|||=||u||\text{ and }f(u)=||u||^{2}.
\]
}
\end{cor}

\begin{proof}
\textcolor{black}{Use Cor.~\ref{cor:extension=000020continuous}
with $\mathcal{F}=[\mathbb{R}u_{\varepsilon}]$ where $[u_{\varepsilon}]=u$,
$g_{\eps}:\R u_{\eps}\ra\R$, $\lambda u_{\eps}\ra\lambda||u_{\eps}||_{\eps}^{2}$
so that $g([\lambda_{\eps}u_{\eps}])=[\lambda_{\eps}||u||_{\eps}^{2}]$,
so that $|||g|||=||u||$.}
\end{proof}
\textcolor{black}{This result has been already proved in Prop.~3.23
of \cite{Gar05} without using Thm.~\ref{thm:HBT}.}

\subsection{The Hahn-Banach theorem in a Colombeau space based on a family of
locally convex topological vector spaces}

In all this section, we assume that $E:=(E_{\eps},(p_{\eps i})_{i\in\N})$
is a family of locally convex topological vector spaces.
\begin{thm}
\label{thm:HBT=000020LCTVS}Let $\mathcal{F}=[F_{\eps}]\subseteq\mathcal{G}_{E}$
be an internal $\rti$-submodule, and let $g=[g_{\eps}]:\text{\ensuremath{\mathcal{F}\longrightarrow\rti}}$
be an $\rti$-linear map. Assume that there exists $i\in\N$ such
that $g$ is well defined on $\mathcal{F}_{i}$ and 
\begin{equation}
\exists C\in\rti_{>0}\,\forall x\in\mathcal{F}_{i}:\,g(x)\leq Cp_{i}(x).\label{eq:bound-condition-1}
\end{equation}
Then, there exists an $\rti$-linear map $f=[f_{\eps}]:\mathcal{G}_{E_{i}}\ra\rti$
that extends $g$ and satisfies 
\[
\forall x\in\mathcal{G}_{E_{i}}:\,f(x)\leq Cp_{i}(x).
\]
\end{thm}

\begin{proof}
Consider the following two properties
\begin{enumerate}
\item \label{enu:HBT-LCVTS1}for any representative $(g_{\eps})$ of $g$,
for any representative $(C_{\eps})$ of $C$, for all $q\in\N$, there
exists $\eta_{q}$ 
\[
\forall\eps\in(0,\eta_{q}]\,\forall x\in(F_{\eps}\backslash\mathrm{Ker}(p_{\eps i}))\cup\{0\}:\,g_{\eps}(x)\leq(C_{\eps}+\rho_{\eps}^{q})p_{\eps i}(x)
\]
\item \label{enu:HBT-LCVTS2}for any representative $(g_{\eps})$ there
exists a representative $(C'_{\eps})$ of $C$ such that 
\[
\forall^{0}\eps\,\forall x\in(F_{\eps}\backslash\mathrm{Ker}(p_{\eps i}))\cup\{0\}:\,g_{\eps}(x)\leq C'_{\eps}p_{\eps i}(x).
\]
\end{enumerate}
Similarly to Prop.~\ref{prop:reduce}, we have that \eqref{eq:bound-condition-1}$\Rightarrow$\ref{enu:HBT-LCVTS1}$\Rightarrow$\ref{enu:HBT-LCVTS2}.
Indeed, if \ref{enu:HBT-LCVTS1} does not hold, then there exist a
representative $(g_{\eps})$ of $g$, a representative $(C_{\eps})$
of $C$, a sequence $(\eps_{k})$ non-increasing and converging to
$0$, a sequence $(x_{\eps_{k}})$ such that 
\[
\forall k\in\N:\,x_{\eps_{k}}\in(F_{\eps_{k}}\backslash\mathrm{Ker}(p_{\eps_{k}i}))\cup\{0\},\,\,\,g_{\eps_{k}}(x_{\eps_{k}})>(C_{\eps_{k}}+\rho_{\eps_{k}}^{q})p_{\eps_{k}i}(x_{\eps_{k}}).
\]
Clearly, $x_{\eps_{k}}\neq0$ for all $k$, so we can assume that
$p_{\eps_{k}i}(x_{\eps_{k}})=1$ for all $k$. It follows that $[x_{\eps_{k}}]|_{L}\in\mathcal{F}_{i}|_{L}$
and $[g_{\eps_{k}}(x_{\eps_{k}})]|_{L}\geq[C_{\eps_{k}}+\rho_{\eps_{k}}^{q}]|_{L}$
where $L:=\{\eps_{k}\mid k\in\N\}\subseteq_{0}I$, which contradicts
\eqref{eq:bound-condition-1}. Therefore, the first implication is
proved. The implication \ref{enu:HBT-LCVTS1}$\Rightarrow$\ref{enu:HBT-LCVTS2}
is simple and can be proved exactly as in Prop~\ref{prop:reduce}.
Without loss of generality we assume that $\mathrm{Ker}(p_{\eps i})\subseteq F_{\eps}$
for all $\eps$. We claim that 
\begin{equation}
\forall^{0}\eps\,\forall x\in\mathrm{Ker}(p_{\eps i})\backslash\{0\}:\,g_{\eps}(x)=0.\label{eq:3.13}
\end{equation}
Indeed, if the claim does not hold then there exist a sequence $(\eps_{k})$
non-increasing and converging to $0$ and a sequence $(x_{\eps_{k}})$
such that 
\[
\forall k\in\N:\,x_{\eps_{k}}\in\mathrm{Ker}(p_{\eps_{k}i})\backslash\{0\},\,\,\,g_{\eps_{k}}(x_{\eps_{k}})\neq0.
\]
It follows that $y_{\eps_{k}}:=x_{\eps_{k}}/g_{\eps_{k}}(x_{\eps_{k}})\in\mathrm{Ker}(p_{\eps_{k}i}),$
and hence $[y_{\eps_{k}}]|_{S}\in\mathcal{F}_{i}|_{S}$ where $S:=\{\eps_{k}\mid k\in\N\}\subseteq_{0}I$.
Thus, using \eqref{eq:bound-condition-1} with $[y_{\eps_{k}}]_{S}$
leads to a contradiction. Therefore, the claim holds. Using now \ref{enu:HBT-LCVTS2}
and \eqref{eq:3.13} we get 
\[
\forall^{0}\eps\,\forall x\in F_{\eps}:\,g_{\eps}(x)\leq C'_{\eps}p_{\eps i}(x).
\]
We use now the classical Hahn-Banach theorem to conclude exactly as
we did in the proof of Thm.~\ref{thm:HBT}.
\end{proof}
\begin{rem}
~
\begin{enumerate}
\item Clearly, Thm.~\ref{thm:HBT=000020LCTVS} holds when the $i$ is replaced
by any finite subset $J\subseteq\N$.
\item Since any element of $\mathcal{G}_{E}$ defines an element of $\mathcal{G}_{E_{J}}$
for all $J\subseteq\N$, the extension map $f$ of $g$ in Thm.~\ref{thm:HBT=000020LCTVS}
is also defined on $\mathcal{G}_{E}$.
\end{enumerate}
\end{rem}

\section{\textcolor{black}{Geometric forms of the Hahn-Banach theorem: Separation
of convex sets}}

As it is well known, one of the classical applications of the Hahn-Banach
theorem in convex geometry is the separation of disjoint and convex
sets in a normed vector space $E$ by a closed hyperplane. The main
goal of this section is to generalize this property in spaces of nonlinear
generalized functions.

The following definitions are generalizations of the classical definitions
of hyperplane, convex set, disjoint sets and non-empty set of an $\rti$-module
$\mathcal{G}$.

\textcolor{black}{In all this section $E:=(E_{\eps},||\cdot||_{\eps})$
is a family of normed spaces.}
\begin{defn}
\label{def:basic-def}Let $\mathcal{G}$ be an arbitrary $\rti$-module,
$A$, $B\subseteq\mathcal{G}$ be two subsets, $\alpha\in\rti$, and
let $f:\mathcal{G}\longrightarrow\rti$ be an $\rti$-linear map
\begin{enumerate}
\item \textcolor{black}{\label{enu:non-identically=000020map=000020}We
say that $f$ is }\textcolor{black}{\emph{non-identically null}}\textcolor{black}{{}
if 
\[
\exists x\in\mathcal{G}:\,|f(x)|>0.
\]
We recall that in $\rti$, $x>0$ means $x\ge0$ and $x$ invertible.
If $\mathcal{G}$ is a Colombeau space $\mathcal{G}_{E}$ based on
a family of normed spaces, and $f=[f_{\varepsilon}]$, then $f$ being
non identically null implies that $f_{\varepsilon}$ is non-identically
null for every $\varepsilon$ small, i.e.~$\forall^{0}\eps\,\exists x\in E_{\eps}:\,f_{\eps}(x)\ne0$.}
\item \textcolor{red}{\label{enu:hyperplane}}\textcolor{black}{A }\textcolor{black}{\emph{hyperplane}}\textcolor{black}{{}
$\tilde{H}$ of $\mathcal{G}$ is a set of the form 
\[
\tilde{H}=\{u\in\mathcal{G}\mid f(u)=\alpha\}
\]
where $f$ is linear and non-identically null.}
\item \label{enu:convexity}We say that $A$ is \emph{convex} if 
\[
\forall x,y\in A\,\forall t\in[0,1]\subseteq\rti:\,tx+(1-t)y\in A.
\]
\item \label{enu:strongly=000020disjoint=000020sets}We say that $A$ and
$B$ are \emph{strongly disjoint} if 
\begin{equation}
\forall L\subseteq_{0}I:\,A|_{L}\cap B|_{L}=\emptyset,\label{eq:strongly_disjoint_sets}
\end{equation}
or equivalently, 
\[
\forall e\in\rti:\,e^{2}=e,\,e\neq0\Rightarrow eA\cap eB=\emptyset
\]
We recall that for $L$, $K\subseteq I$, we say $L\subseteq_{0}K$
if $L\subseteq K$ and $0\in\overline{L}$.
\item \label{enu:no-empty=000020set}$A$ is said to be \emph{strongly non-empty}
if 
\begin{equation}
\forall L\subseteq_{0}I:\,A|_{L}\neq\emptyset
\end{equation}
\end{enumerate}
\end{defn}

\begin{example}
\label{exa:example}Let $\mathcal{G}=\mathcal{G}_{E}$.
\begin{enumerate}
\item \label{enu:ex13}If $A=[A_{\varepsilon}]$ or $A=\left\langle A_{\varepsilon}\right\rangle $
or $A=[A_{\varepsilon}]|_{L}+\left\langle A_{\varepsilon}\right\rangle |_{L^{c}}$
where $L$, $L^{c}\subseteq_{0}I$ and $A_{\varepsilon}\subseteq E_{\varepsilon}$
is convex for every $\varepsilon$ small then $A$ is convex.
\item Let $(A_{\eps})$ and $(B_{\eps})$ be two nets of convex sets of
$\prod_{\eps\in I}E_{\eps}$. Then $[A_{\eps}]\cup[B_{\eps}]$ is
not necessary convex even if $A_{\eps}\cup B_{\eps}$ is convex for
all $\eps$ small. As an example, take $L\subset I$ such that $0\in\overline{L}$,
$0\in\overline{L^{c}}$ and set 
\[
\forall\varepsilon\in L:\,A_{\varepsilon}:=[0,2],\,B_{\eps}:=[2,3],\,\text{and }\forall\varepsilon\in L^{c}:\,A_{\eps}:=[0,1],\,B_{\eps}:=[1,3].
\]
Taking $x=1\in A$, $y=2\in B$ but $\frac{1}{2}x+\frac{1}{2}y\not\in[A_{\eps}]\cup[B_{\eps}]$.
\item \textcolor{black}{If $A\subseteq\R\subseteq\rti$ then $A$ is not
necessary convex even if $A$ is convex in $\R$.}
\end{enumerate}
\end{example}

We have seen in \ref{enu:ex13} of Example \ref{exa:example} that
if an internal or a strongly internal subset $A\subseteq\mathcal{G}_{E}$
has a convex representative then $A$ is convex. One could ask if
the reciprocal is true. The answer is given in the following lemma
\begin{lem}
\textcolor{black}{\label{lem:conv=000020repr}Let $A$ be a non-empty
convex subset of $\mathcal{G}_{E}$ and let $x$, $y\in A$.}
\begin{enumerate}
\item \label{enu:=000020strongly=000020inter.=000020conx=000020rep}Assume
$A=\sint{A_{\eps}}$ is a strongly internal set. Then 
\begin{equation}
\forall[x_{\eps}]=x\,\forall[y_{\eps}]=y\,\forall^{0}\eps\,\forall t\in[0,1]:\,x_{\eps}+t(y_{\eps}-x_{\eps})\in A_{\eps}.\label{eq:strongly_inter_conx_rep}
\end{equation}
\item \label{enu:internal=000020conv.=000020rep}Assume $A$ is a sharply
bounded internal set and let $(A_{\eps})$ be a sharply bounded representative
of $A$. Then, there exists a sharply bounded representative $(\tilde{A}_{\eps})$
such that 
\begin{equation}
\forall^{0}\eps\,\forall u,v\in A_{\eps}\,\forall\lambda\in[0,1]:\,u+\lambda(v-u)\in\tilde{A}_{\eps}\label{eq:2.31}
\end{equation}
In other words, for all $\eps$ small, any segment whose endpoints
belongs to $A_{\eps}$ is included in $\tilde{A}_{\eps}$. 
\item \textcolor{black}{\label{enu:internal=000020conx=000020repr=000020finite=000020dim.}If
$A$ is a sharply bounded internal set and if there exists $n\in\N$
such that $\forall^{0}\eps:\text{dim}(E_{\eps})\leq n$, then A has
a convex representative.}
\end{enumerate}
\end{lem}

\begin{proof}
\textcolor{black}{\ref{enu:=000020strongly=000020inter.=000020conx=000020rep}:
}First note that since $A$ is a strongly internal set, for all representatives
$[x_{\eps}]=x$, $[y_{\eps}]=y$, and $[t_{\eps}]=t\in[0,1]$, there
exists $\eps_{0}$ such that $x_{\eps}+t_{\eps}(y_{\eps}-x_{\eps})\in A_{\eps}$
for all $\eps\in(0,\eps_{0}]$ but where, unlike \eqref{eq:strongly_inter_conx_rep},
$\eps_{0}$ depends on $t$.

\textcolor{black}{If \eqref{eq:strongly_inter_conx_rep} is not satisfied
then for some representatives $[x_{\eps}]=x$, $[y_{\eps}]=y$, there
exists a non-increasing sequence $(\varepsilon_{k})_{k}$ that tends
to $0$, and $t_{\eps_{k}}\in[0,1]$ such that 
\begin{equation}
\forall k\in\mathbb{N}:\,x_{\varepsilon_{k}}+t_{\varepsilon_{k}}(y_{\varepsilon_{k}}-x_{\eps_{k}})\not\in A_{\varepsilon_{k}}.\label{eq:weak=000020conv2}
\end{equation}
Set $\lambda_{\eps}:=t_{\eps_{k}}$ if $\eps=\eps_{k}$ and $\lambda_{\eps}:=0$
otherwise. Then, since $A$ is a strongly internal set, every representative
of $x+\lambda(y-x)$ belongs to $A_{\eps}$ for $\eps$ small, which
contradicts \eqref{eq:weak=000020conv2} at $\eps=\eps_{k}$.}

\textcolor{black}{\ref{enu:internal=000020conv.=000020rep}: Note
first that when $A$ is a sharply bounded internal set, then it has
a sharply bounded representative (see Lem.~3.3 of \cite{ObVe08}).
Let $(A_{\eps})$ be a sharply bounded representative of $A$. We
claim now that 
\begin{multline}
\forall q\in\N\,\exists\eta_{q}\in I\,\forall\eps\in(0,\eta_{q}]\,\forall u,v\in A_{\eps}\\
\,\forall\lambda\in[0,1]:\,u+\lambda(v-u)\in A_{\eps}+B_{\rho_{\eps}^{q}}(0).\label{eq:2.30}
\end{multline}
Indeed, assume that \eqref{eq:2.30} does not hold, then there exists
$q\in\N$, there exists a non-increasing sequence $(\eps_{k})_{k}$
of $I$ that tends to $0$, there exist $(u_{\eps_{k}})_{k},$$(v_{\eps_{k}})_{k}$,
$(\lambda_{\eps_{k}})_{k}$ where $u_{\eps_{k}},v_{\eps_{k}}\in A_{\eps_{k}}$
and $\lambda_{\eps_{k}}\in[0,1]$ for all $k\in\N$ such that 
\[
u_{\eps_{k}}+\lambda_{\eps_{k}}(v_{\eps_{k}}-u_{\eps_{k}})\not\in A_{\eps_{k}}+B_{\rho_{\eps_{k}}^{q}}(0).
\]
It follows that 
\begin{equation}
\forall k\in\N:\,B_{\rho_{\eps_{k}}^{q+1}}(u_{\eps_{k}}+\lambda_{\eps_{k}}(v_{\eps_{k}}-u_{\eps_{k}}))\subseteq A_{\eps_{k}}^{c}.\label{eq:2.33}
\end{equation}
Take $u=[u_{\eps}]$, $v=[v_{\eps}]\in A$ and $t=[t_{\eps}]\in[0,1]$
with $u_{\eps}=u_{\eps_{k}}$, $v_{\eps}=v_{\eps_{k}}$ and $t_{\eps}=\lambda_{\eps_{k}}$
when $\eps=\eps_{k}$. This is possible since the representative $(A_{\eps})$
is sharply bounded. Since $A$ is convex and internal, there exists
a net $(z_{\eps})\in\mathcal{N}_{E}$ such that 
\[
u_{\eps_{k}}+\lambda_{\eps_{k}}(v_{\eps_{k}}-u_{\eps_{k}})+z_{\eps_{k}}\in A_{\eps_{k}}
\]
for $k\in\N$ sufficiently large, but this contradicts \eqref{eq:2.33},
and hence \eqref{eq:2.30} holds. Set 
\[
\tilde{A}_{\eps}:=A_{\eps}+B_{\rho_{\eps}^{q}}(0),\,\,\forall\eps\in(\eta_{q+1},\eta_{q}]
\]
where $(\eta_{q})_{q}$ is given by \eqref{eq:2.30} and it is assumed
to be non-increasing and converging to $0$. Clearly, the net $(\tilde{A}_{\eps})$
is sharply bounded. Since the net $n_{\eps}:=\rho_{\eps}^{q}$ when
$\eps\in(\eta_{q+1},\eta_{q}]$ is negligible, we have $(d_{H}(\tilde{A}_{\eps},A_{\eps}))_{\eps}\in\mathcal{N}_{\R}$,
where $d_{H}(\cdot,\cdot)$ denotes the Hausdorff distance. Thus,
by Thm.~3.8 of \cite{ObVe08}, we have $[\tilde{A}_{\eps}]=[A_{\eps}]$.
Therefore,} \ref{enu:internal=000020conv.=000020rep}\textcolor{black}{{}
is proved.}

\ref{enu:internal=000020conx=000020repr=000020finite=000020dim.}:
See \cite[Lem.~17]{GKV15}.
\end{proof}
In the next theorem, we give an assumption under which a convex strongly
internal set has a representative consisting of convex sets
\begin{thm}
\label{thm:ConRep}Let $A\subseteq\mathcal{G}_{E}$ be a non-empty
strongly internal convex set. Then, A has a representative $(\tilde{A}_{\eps})$
consisting of convex sets if and only if it has a representative $(A_{\eps})$
satisfying the following property 
\begin{equation}
\forall q\in\N\,\forall^{0}\eps\,\forall C\subseteq E_{\eps}:\,C+B_{\rho_{\eps}^{q}}(0)\subseteq A_{\eps}\cap B_{\rho_{\eps}^{-q}}(0)\ \Rightarrow\ \mathrm{conv}(C)\subseteq A_{\eps}\label{eq:4.10}
\end{equation}
where $\mathrm{conv}(C)$ stands for the convex hull of the set $C$.
Moreover, if the representative $(A_{\eps})$ is sharply bounded,
the representative $(\tilde{A}_{\eps})$ is sharply bounded as well.
\end{thm}

\begin{proof}
Clearly, a representative of $A$ consisting of convex sets satisfies
\eqref{eq:4.10}. We prove now that \eqref{eq:4.10} is sufficient
to have a representative consisting of convex sets. Setting 
\[
\forall\eps\in I\,\forall q\in\N:\,C_{\eps q}:=\left(\left(A_{\varepsilon}\cap B_{\rho_{\eps}^{-q}}(0)\right)^{c}+B_{2\rho_{\eps}^{q}}(0)\right)^{c}.
\]
Since $A$ is a non-empty strongly internal set, for all $q$ sufficiently
large, the set $C_{\eps q}$ is non-empty for all $\eps$ small. By
construction we have that $C_{\eps q}+B_{\rho_{\eps}^{q}}(0)\subseteq A_{\varepsilon}\cap B_{\rho_{\eps}^{-q}}(0)$.
Thus, by \eqref{eq:4.10}, 
\[
\forall q\in\N\,\exists\eta_{q}\in I\,\forall\eps\in(0,\eta_{q}]:\,A_{\eps q}:=\mathrm{conv}\left(C_{\eps q}\right)\subseteq A_{\eps}.
\]
Without loss of generality, we assume that the sequence $(\eta_{q})$
in non-increasing and converging to $0$, and set $\tilde{A}_{\eps}:=A_{\eps q}$
whenever $\eps\in(\eta_{q+1},\eta_{q}]$. Setting also $n_{\eps}:=\rho_{\eps}^{q}$
whenever $\eps\in(\eta_{q+1},\eta_{q}]$. We clearly have $\langle\tilde{A}_{\eps}\rangle\subseteq A$.
Let now $x=[x_{\eps}]\in A$. By property \emph{(iv)} of Thm.~10
in \cite{GKV24}, there exists $p\in\N$ 
\[
\forall^{0}\eps:\,d\left(x_{\eps},\left(A_{\varepsilon}\cap B_{n_{\eps}^{-1}}(0)\right)^{c}\right)\geq\rho_{\eps}^{p}>2n_{\eps}
\]
which implies that $x_{\eps}\in\tilde{A}_{\eps}$ for all $\eps$
small. Since the latter holds for any representative $[x_{\eps}]$
of $x$, the inclusion $A\subseteq\langle\tilde{A}_{\eps}\rangle$
holds. Therefore, the proof of the first part is completed. The second
part follows easily from the construction of the representative $(\tilde{A}_{\eps})$
since we have $\tilde{A}_{\eps}\subseteq A_{\eps}$ for all $\eps$
small.
\end{proof}
We show in the next proposition that the property \eqref{eq:4.10}
holds when $\mathrm{dim}(E_{\eps})\leq n$ for all $\eps$ small.
\begin{prop}
Let $A=\langle A_{\eps}\rangle\subseteq\mathcal{G}_{E}$ be a non-empty
convex subset, and assume that there exists $n\in\N$ such that $\mathrm{dim}(E_{\eps})\leq n$
for all $\eps$ small. Then, property \eqref{eq:4.10} holds.
\end{prop}

\begin{proof}
If the property does not hold, there exist $q\in\N$, a sequence $(\eps_{k})$
non-increasing and converging to $0$, a sequence $(C_{\eps_{k}})$
of subsets $C_{\eps_{k}}\subseteq E_{\eps_{k}}$, a sequence $(x_{\eps_{k}})$
of elements $x_{\eps_{k}}\in\mathrm{conv}(C_{\eps_{k}})$ such that
\begin{equation}
\forall k\in\N:\,C_{\eps_{k}}+B_{\rho_{\eps_{k}}^{q}}(0)\subseteq A_{\eps_{k}}\cap B_{\rho_{\eps_{k}}^{-q}}(0)\,\,\,\text{and }\,\,\,x_{\eps_{k}}\not\in A_{\eps_{k}}.\label{eq:4.13}
\end{equation}
By Carathéodory's theorem, for all $k\in\N$, there exist $\lambda_{0\eps_{k}},...,\lambda_{n\eps}\in[0,1]\subseteq\R$,
$y_{0\eps_{k}},...,y_{n\eps_{k}}\in C_{\eps_{k}}$ such that 
\begin{equation}
\sum_{j=0}^{n}\lambda_{j\eps_{k}}=1\,\,\,\text{and}\,\,\,x_{\eps_{k}}:=\sum_{j=0}^{n}\lambda_{j\eps_{k}}y_{j\eps_{k}}.\label{eq:4.14}
\end{equation}
From \eqref{eq:4.13}, $y_{j}:=[y_{j\eps_{k}}]|_{L}\in\langle A_{\eps}\rangle|_{L}$
for all $j=0,...,n$, where $L:=\{\eps_{k}\mid k\in\N\}\subseteq_{0}I$.
From \eqref{eq:4.14} and the convexity of $A$, we have that 
\[
x:=[x_{\eps_{k}}]|_{L}=\sum_{j=0}^{n}[\lambda_{j\eps_{k}}]|_{L}y_{j}\in A|_{L}.
\]
Since $A$ is a strongly internal set, we have that $x_{\eps_{k}}\in A_{\eps_{k}}$
for all $k$ sufficiently large, but this contradicts \eqref{eq:4.13}.
Therefore, property \eqref{eq:4.10} holds.
\end{proof}
We now equivalently reformulate $\eps$-wise the property of being
strongly disjoint:
\begin{prop}
\textcolor{black}{\label{prop:strong=000020disjoint=000020}Let $A=[A_{\varepsilon}]$
and $B=\langle B_{\varepsilon}\rangle$ be two non-empty subsets of
$\mathcal{G}_{E}$. Then $A$ and }\textbf{\textcolor{black}{$B$}}\textcolor{black}{{}
are strongly disjoint if and only if for any representative $(A_{\varepsilon})$
of $A$ there exists a representative $(B_{\eps})$ of $B$ such that
\begin{equation}
\forall^{0}\varepsilon:\,A_{\varepsilon}\cap B_{\varepsilon}=\emptyset.\label{eq:inters}
\end{equation}
}
\end{prop}

\begin{proof}
\textcolor{black}{We prove first that the condition is sufficient.
Let $L\subseteq_{0}I$ and let $x\in_{L}A$, we have to prove that
$x\notin_{L}B$. Take any representative }$(A_{\eps})$ of $A$, there
exists a representative $(x_{\varepsilon})$ of $x$ such that $\forall\varepsilon\in L:\,x_{\varepsilon}\in A_{\varepsilon}$.
By assumption, there exists a representative $(B_{\eps})$ of $B$
such that \textcolor{black}{\eqref{eq:inters}. In particular, $\forall^{0}\varepsilon\in L:\,x_{\varepsilon}\not\in B_{\varepsilon}$,
i.e.~$x\notin_{L}B$.}

Assume now that $A$ and $B$ are strongly disjoint. We claim that
for any representative $(A_{\varepsilon})$ of $A$, for any representative
$(B_{\eps})$ of $B$, for all $q\in\N$ 
\begin{equation}
\exists\eta_{q}\in I\,\forall\varepsilon\in(0,\eta_{q}]:\,\left(\left(B_{\varepsilon}\cap B_{\rho_{\eps}^{-q}}(0)\right)^{c}+B_{\rho_{\eps}^{q}}(0)\right)^{c}\cap A_{\varepsilon}=\emptyset\label{eq:weak=000020intersection}
\end{equation}
Indeed, assume that \eqref{eq:weak=000020intersection} does not hold.
Then, there exist a representative $(A_{\varepsilon})$ of $A$, a
representative $(B_{\varepsilon})$ of $B$, $q\in\mathbb{N},$ a
non-increasing sequence $(\varepsilon_{k})_{k}$ that tends to $0$
and a sequence $(x_{\varepsilon_{k}})_{k}$ of elements $x_{\eps_{k}}\in E_{\eps_{k}}$such
that 
\begin{equation}
\forall k\in\mathbb{N}:\,x_{\varepsilon_{k}}\in\left(\left(B_{\varepsilon_{k}}\cap B_{\rho_{\eps_{k}}^{-q}}(0)\right)^{c}+B_{\rho_{\varepsilon_{k}}^{q}}(0)\right)^{c}\cap A_{\varepsilon_{k}}.\label{eq:3.58}
\end{equation}
Take $u=[u_{\varepsilon}]\in\mathcal{G}_{E}$ where $u_{\varepsilon}=x_{\varepsilon_{k}}$
when $\varepsilon=\varepsilon_{k}$. Note that this is possible since
$((B_{\varepsilon_{k}}\cap B_{\rho_{\eps_{k}}^{-q}}(0))^{c}+B_{\rho_{\varepsilon_{k}}^{q}}(0))^{c}\subseteq B_{\rho_{\eps_{k}}^{-q}}(0)$
and hence $\left|x_{\eps_{k}}\right|\leq\rho_{\eps_{k}}^{-q}$ for
all $k\in\N$. It follows from \eqref{eq:3.58} that $u\in_{L}A$
where $L=\left\{ \eps_{k}\mid k\in\N\right\} \subseteq_{0}I$. On
the other hand, we have 
\[
\forall k\in\mathbb{N}:\,d(x_{\eps_{k}},B_{\varepsilon_{k}}^{c})\geq d(x_{\eps_{k}},(B_{\varepsilon_{k}}\cap B_{\rho_{\eps_{k}}^{-q}}(0))^{c})>\rho_{\eps_{k}}^{q+1}.
\]
By Thm.~10 of \cite{GKV24}, $u\in_{L}B$, which contradicts the
fact that $A$ and $B$ are strongly disjoint. Therefore, our claim
is proved. Without loss of generality we can assume that the sequence
$(\eta_{q})$ is non-increasing and converging to $0$. Set
\[
\forall\varepsilon\in(\eta_{q+1},\eta_{q}]:\,B'_{\varepsilon}:=\left(\left(B_{\varepsilon}\cap B_{\rho_{\eps}^{-q}}(0)\right)^{c}+B_{\rho_{\eps}^{q}}(0)\right)^{c}.
\]
From \eqref{eq:weak=000020intersection}, we have 
\[
\forall\eps\in(0,\eta_{0}]:\,B_{\eps}'\cap A_{\eps}=\emptyset.
\]
It remains to prove that $(B'_{\eps})$ is a representative of $B$.
Clearly, since $\forall\eps:B'_{\eps}\subseteq B_{\eps}$, we have
$\langle B'_{\eps}\rangle\subseteq\langle B{}_{\eps}\rangle$. Let
now $x=[x_{\eps}]\in\langle B{}_{\eps}\rangle$. By Thm.~10. of \cite{GKV24},
there exists $p\in\N$ such that $\forall^{0}\eps:\,d(x_{\eps},B_{\eps}^{c})>\rho_{\eps}^{p}$,
and since the net $n_{\varepsilon}:=\rho_{\varepsilon}^{q}$ when
$\varepsilon\in(\eta_{q+1},\eta_{q}]$, is negligible we get $d(x_{\eps},(B_{\varepsilon}\cap B_{n_{\eps}^{-1}}(0))^{c})>\rho_{\eps}^{q}>n_{\eps}$
for all $\eps$ small. Hence $x_{\eps}\in B'_{\eps}$ for all $\eps$
small. Since $[x_{\eps}]$ is an arbitrary representative of $x$,
we have that $x\in\langle B'_{\eps}\rangle$. Therefore, the net $(B'_{\eps})$
is a representative of $B$ which completes the proof.
\end{proof}
\textcolor{black}{We r}ecall from \cite{GiKu18} that the strong exterior
of a set $A\subseteq\mathcal{G}$ is defined by 
\begin{equation}
\text{ext}(A):=\{x\in\mathcal{G}\mid\forall a\in A:\,|a-x|>0\}\label{eq:exterior}
\end{equation}
or equivalently\footnote{The result is stated for the special case $\mathcal{G}=\rti$, but
the statement holds for any $\rti$-module $\mathcal{G}$ } (see \cite[Lem.~3.4]{GiKu18})
\[
\text{ext}(A)=\{x\in\mathcal{G}\mid\forall S\subseteq I:e_{S}\neq0\ \Rightarrow\ xe_{S}\not\in A\}.
\]

\begin{prop}
\textcolor{black}{\label{hyperplane}Let $(\mathcal{G},||\cdot||)$
be an $\rti$-normed $\rti$-module, $f:\mathcal{G}\longrightarrow\rti$
be a non-identically null $\rti$-linear map, $\alpha\in\rti$ and
let $\tilde{H}\subseteq\mathcal{G}$ be a hyperplane defined by 
\[
\tilde{H}=\{x\in\mathcal{G}\mid f(x)=\alpha\}=:\{f=\alpha\}.
\]
Then, the following properties hold:}
\begin{enumerate}
\item \textcolor{magenta}{\label{enu:hyperplane1}}\textcolor{black}{If
$f$ is continuous (in the sharp topology), then $\tilde{H}$ is closed
in the sharp topology.}
\item \textcolor{magenta}{\label{enu:hyperplane2}}\textcolor{black}{If
$\mathrm{ext}(\tilde{H})$ is a sharply open set}\footnote{or equivalently $\tilde{H}$ is sharply closed e.g.~when $\tilde{H}$
is an internal set, see \cite[Lem.~3.5]{GiKu18}}\textcolor{black}{, then $f$ is continuous.}
\item \textcolor{magenta}{\label{enu:hyperplane3}}\textcolor{black}{If
$\mathcal{G}=\mathcal{G}_{E}$ and $f=[f_{\varepsilon}]$. Then 
\begin{equation}
\tilde{H}=[\{u\in E_{\varepsilon}\mid f_{\varepsilon}(u)=\alpha_{\varepsilon}\}].\label{eq:internal_hyperplane}
\end{equation}
where the right hand side depends neither on the representative $(f_{\eps})$
nor on the representative $(\alpha_{\eps})$ of $\alpha$. In particularly,
$\tilde{H}$ is closed.}
\end{enumerate}
\end{prop}

\begin{proof}
\textcolor{black}{\ref{enu:hyperplane1}: Assume that $f$ is continuous
(in the sharp topology) and let $(u_{n})_{n}$ be a sequence of elements
of $\tilde{H}$ that converges to $u$ in the sharp topology, that
is 
\[
\forall q\in\mathbb{R}_{>0}\,\exists N\in\mathbb{N}\,\forall n>N:\,||u_{n}-u||<\diff\rho^{q}.
\]
Since $f$ is continuous we have 
\[
\forall n>N:\,|f(u_{n})-f(u)|\leq C||u_{n}-u||<C\diff\rho^{q}
\]
where $C\in\rti_{\geq0}$ (see }Sec.~\ref{subsec:LinMaps}\textcolor{black}{).
Since $(u_{n})_{n}\in\tilde{H}$, $|\alpha-f(u)|<C\diff\rho^{q}$
for all $q\in\N$. Letting $q\to+\infty$, we obtain $f(u)=\alpha.$
Thus, $u\in\tilde{H}$ and hence $\tilde{H}$ is closed.}

\textcolor{black}{\ref{enu:hyperplane2}: Since $f$ is non-identically
null, $\text{ext}(\tilde{H})$ is non-empty. Let $v_{0}\in\mathrm{ext}(\tilde{H})$.
Then $\left|f(v_{0})-\alpha\right|>0$ by }\eqref{eq:exterior}\textcolor{black}{.
It follows that there exists $L\subseteq I$ such that $e_{L}f(v_{0})<e_{L}\alpha$
and $e_{L^{c}}f(v_{0})>e_{L^{c}}\alpha$. Since $f$ is non-identically
null, there exists $w\in\mathcal{G}$ such that $f(w)=1$. Set $u_{0}:=e_{L}v_{0}+e_{L^{c}}(-v_{0}+2\alpha w)$.
It follows that $f(u_{0})<\alpha$, and hence $u_{0}\in\text{ext}(\tilde{H})$.
By assumption, $\text{ext}(\tilde{H})$ is a sharply open set. Thus,
there exists $r\in\rti_{>0}$ such that $B_{r}(u_{0})\subseteq\text{ext}(\tilde{H})$.
We claim that 
\begin{equation}
\forall v\in B_{r}(u_{0}):\,f(v)<\alpha.\label{eq:3.23}
\end{equation}
Indeed, assume that there exists $u_{1}\in B_{r}(u_{0})$ and $\exists L_{1}\subseteq_{0}I$
such that $\alpha<_{L_{1}}f(u_{1})$. Since $B_{r}(u)$ is convex,
we have 
\begin{equation}
\{u_{\lambda}=(1-\lambda)u_{0}+\lambda u_{1}\mid\lambda\in[0,1]\}\subseteq B_{r}(u)\subset\text{ext}(\tilde{H})\label{eq:3.30}
\end{equation}
which implies that $|f(u_{\lambda})-\alpha|>0$ for all $\lambda\in[0,1]$.
On the other hand, for $\lambda=_{L_{1}}\frac{f(u_{1})-\alpha}{f(u_{1})-f(u_{0})}$
we have $f(u_{\lambda})=_{L_{1}}\alpha$ which contradicts \eqref{eq:3.30}.
Consequently, \eqref{eq:3.23} holds. For $v\in B_{r}(u_{0})$, we
have $v=u_{0}+rz$ where $z\in B_{1}(0)$. It follows that 
\[
\forall z\in B_{1}(0):\,f(u_{0}+rz)<\alpha
\]
which yields
\[
\forall z\in B_{1}(0):\,f(z)<\frac{1}{r}(\alpha-f(u_{0})).
\]
By a homothety argument we have
\[
\forall w\in\mathcal{G}:\,f(w)\leq\frac{1}{r}(\alpha-f(u_{0}))||w||.
\]
Thus, $f$ is continuous (see }Sec.~\ref{subsec:LinMaps}\textcolor{black}{).}

\textcolor{black}{\ref{enu:hyperplane3}: Let $(f_{\eps})$ and $(\alpha_{\eps})$
be any representatives of $f$ and of $\alpha$ respectively. In order
to show that 
\[
\tilde{H}\subseteq[\{u\in E_{\varepsilon}\mid f_{\varepsilon}(u)=\alpha_{\varepsilon}\}],
\]
we have to prove that for every $x\in\tilde{H}$, there exists a representative
$(x'_{\eps})$ of $x$ that satisfies 
\begin{equation}
\forall^{0}\varepsilon:\,f_{\varepsilon}(x'_{\varepsilon})=\alpha{}_{\varepsilon}.\label{eq:3.28}
\end{equation}
Let $x\in\tilde{H}$, so that $f(x)=\alpha$ and, for every representative
$(x_{\eps})$ of $x$, there exists a representative $(\alpha'_{\eps})$
of $\alpha$ such that $\forall^{0}\eps:\,f_{\eps}(x_{\eps})=\alpha_{\eps}'$.
Since $f$ is non-identically null, there exists $v\in\mathcal{G}_{E}$
such that $|f(v)|>0$. Without loss of generality, we can assume that
$f(v)=1$. It follows that for every representative $(v_{\varepsilon})$
of $v$ there exists a negligible net $(z_{\varepsilon})$ such that
$\forall^{0}\varepsilon:\,f_{\varepsilon}(v_{\varepsilon})=1+z_{\varepsilon}$.
Setting $\forall^{0}\eps:\,v_{\eps}':=v_{\varepsilon}/(1+z_{\varepsilon})$
we obtain $\forall^{0}\varepsilon:\,f_{\varepsilon}(v'_{\varepsilon})=1.$
Setting $\forall^{0}\eps:\,x'_{\varepsilon}=x_{\varepsilon}+(\alpha_{\varepsilon}-\alpha'_{\varepsilon})v'_{\varepsilon}$
we obtain \eqref{eq:3.28}. This implies $\tilde{H}\subseteq[\{u\in E_{\varepsilon}\mid f_{\varepsilon}(u)=\alpha_{\varepsilon}\}].$
The opposite inclusion is trivial. Finally, since $\tilde{H}$ is
an internal set and $E_{\varepsilon}$ is a normed space for every
$\varepsilon$, Thm.~3.2 of \cite{ObVe08} implies that $\tilde{H}$
is closed in the sharp topology.}
\end{proof}
\begin{defn}
Let $\mathcal{G}$ be an $\rti$-normed $\rti$-module and let $A$,
$B$ be two subsets of $\mathcal{G}$, $f:\mathcal{G}\ra\rti$ an
$\rti$-linear map non identically null, and let $\alpha\in\rti$.
We say that the hyperplane $\tilde{H}=\{f=\alpha\}$ \emph{separates}
$A$ and $B$ if 
\[
\forall x\in A:\,f(x)\le\alpha\qquad\text{ and }\qquad\forall x\in B:\,f(x)\geq\alpha.
\]
We say that $\tilde{H}$ \emph{strictly separates} $A$ and $B$ if
there exists $r\in\rti_{>0}$ such that 
\[
\forall x\in A:\,f(x)\le\alpha-r\qquad\text{ and }\qquad\forall x\in B:\,f(x)\geq\alpha+r.
\]
\end{defn}

The first geometric form of the Hahn-Banach theorem for normed spaces
states that, given two non-empty, convex and disjoint sets such that
one of these two sets is open, then they can be separated by a closed
hyperplane. We will see in the next example that these conditions
are not sufficient to generalize this property to subsets of $\mathcal{G}_{E}$.
As an example, take $D_{\infty}$, the set of all infinitesimals (see
Sec.~\ref{subsec:Sup}), and take $A=\{[x_{\eps}]\in\rti_{\geq0}\mid\exists c\in\mathbb{R}_{>0}\,\forall^{0}\varepsilon:\,x_{\varepsilon}\geq c\}$.
It is easy to check that $D_{\infty}$ and $A$ are convex and strongly
disjoint sets. Moreover, $D_{\infty}$ is an open set. Assume that
there exists a hyperplane $\tilde{H}=\{f=\alpha\}$ that separates
$A$ and $D_{\infty}$. Then 
\begin{equation}
\forall x\in D_{\infty}\,\forall y\in A:\,f(x)\le\alpha\leq f(y).\label{eq:Separation=000020sets}
\end{equation}
Since $f$ is non identically null, there exists $v\in\rti$ such
that $|f(v)|>0$. Since we have $f(v)=vf(1)$, $f(1)$ is invertible.
Using \eqref{eq:Separation=000020sets} with $x=\diff\rho$ and $y=1$
we get that $f(1)\diff\rho\leq f(1)$. Therefore, we necessary have
$f(1)>0$, which implies that 
\begin{equation}
\forall x\in D_{\infty}\,\forall y\in A:\,x\le\frac{\alpha}{f(1)}\leq y\label{eq:3.24}
\end{equation}
Let $c\in\mathbb{R}_{>0}$ and choose $y=[c]\in A$ and $x=0\in D_{\infty}$
in the last inequality we obtain $0\le\frac{\alpha}{f(1)}\leq c$
for all $c\in\mathbb{R}_{>0}$, which implies that $\frac{\alpha}{f(1)}\approx0$,
and hence $\frac{\alpha}{f(1)}\in D_{\infty}$,which is not possible
since we can choose $x=\frac{\alpha}{f(1)}+\diff\rho\in D_{\infty}$
in \eqref{eq:3.24} yielding to the contradiction $\diff\rho=_{}0$.
Therefore, the sets $D_{\infty}$ and $A$ cannot be separated by
a hyperplane.

The proof of the classical geometric form of the Hahn-Banach theorem
in a normed vector space $E$ is based on the use of the Minkowski
functional $p$ defined by 
\[
\forall x\in E:\,p(x)=\text{inf}\{\alpha\in\mathbb{R}_{>0}\mid x_{0}+\alpha^{-1}(x-x_{0})\in A\}
\]
where $A\subseteq E$ is a convex subset and $x_{0}$ belongs to the
interior of $A$. The existence of the infimum is not always guaranteed
in the non-Archimedean ring $\rti$. So, the definition of the Minkowski
functional cannot be generalized to convex subsets of $\mathcal{G}$
without any additional assumptions.

We start with the following theorem.
\begin{thm}
\label{thm:glb=000020implies=000020inf}Let $A\subseteq\rti$ be a
non-empty strongly open convex set. Assume that the following two
properties holds:
\begin{enumerate}
\item \label{enu:ExtPro}For each $x\in\rti\setminus A$ there exists $L\subseteq_{0}I$
such that $x\in_{L}\mathrm{ext}(A)$;
\item The least upper bound of $A$ exists.
\end{enumerate}
Then the sharp supremum of $A$ also exists.

\end{thm}

\begin{proof}
We denote by $\sigma$ the least upper bound of $A$. Since $A$ is
a non-empty set, there exists $x\in A$, and hence $x\leq\sigma$.
We necessary have $x<\sigma$. Indeed, if $x\not<\sigma$, then by
Lem.~5 of \cite{MTAG21}, there exists $S\subseteq_{0}I$ such that
$x\geq_{S}\sigma$ which, together with the inequality $x\leq\sigma$,
implies that $x=_{S}\sigma$. Since $A$ is a sharply open set, $x+\diff\rho^{q}=_{S}\sigma+\diff\rho^{q}\in_{S}A$
for some $q\in\N$ sufficiently large, which contradicts the fact
that $\sigma$ is the least upper bound. Hence, we necessary have
$x<\sigma$. We claim now that 
\[
\forall t\in[0,1):\,(1-t)x+t\sigma\in A.
\]
Indeed, if the claim does not holds, there exists $\overline{t}\in[0,1)$
such that $\overline{y}:=((1-\overline{t})x+\overline{t}\sigma)\not\in A$.
It follows that 
\begin{equation}
x\leq\bar{y}<\sigma.\label{eq:3.25}
\end{equation}
By \ref{enu:ExtPro}, there exists $L\subseteq_{0}I$ such that $\overline{y}\in_{L}\mathrm{ext}(A)$.
We necessary have $z\leq_{L}\overline{y}$ for all $z\in A$. Indeed,
if there exists $z\in A$ such that $x\leq_{K}\overline{y}<_{K}z$
for some $K\subseteq_{0}L$, the convexity of $A$ implies that $\overline{y}\in_{K}A$
which contradicts the fact that $\overline{y}\in_{L}\mathrm{ext}(A)$.
It follows that $\overline{\sigma}:=e_{L}y+e_{L^{c}}\sigma$ is an
upper bound of $A$ and since $\sigma$ is the least upper bound we
have $\overline{\sigma}\geq\sigma$ and in particularly $\sigma\le_{L}\overline{\sigma}=_{L^{\text{}}}y$
which contradicts \eqref{eq:3.25}. Thus, $(1-t)x+t\sigma\in A$,
$\forall t\in[0,1)$. We prove now that $\sigma$ is the sharp supremum
of $A$. For every $q\in\mathbb{N}$, set $x_{q}=(1-t_{q})x+t_{q}\sigma\in A$
with $t_{q}=1-\diff\rho^{p}(\sigma-x)^{-1}$ where $p>q$ is sufficiently
large so that $t_{q}\in[0,1)$. It follows that $\sigma-\diff\rho^{q}\leq x_{q}$
which proves that $\sigma$ is the supremum of $A$.
\end{proof}
Similarly, if the set $A$ in Thm.~\ref{thm:glb=000020implies=000020inf}
has the greatest lower bound then the infimum of $A$ exists.

We prove in the next proposition that property \ref{enu:ExtPro} of
Thm.~\ref{thm:glb=000020implies=000020inf} holds for any strongly
internal set of $\rti$.
\begin{prop}
\label{prop:ExtProp}Let $A=\langle A_{\eps}\rangle\subseteq\rti$
be a non-empty set. Then, property \ref{enu:ExtPro} of Thm.~\ref{thm:glb=000020implies=000020inf}
holds.
\end{prop}

\begin{proof}
Let $x\in\rti$. Assume that $x\not\in A$. Then, by definition, there
exists a representative $[x_{\eps}]$ of $x$ and $L\subseteq_{0}I$
such that $x_{\eps}\in A_{\eps}^{c}$ for all $\eps\in L$. Thus,
$x\in_{L}[A_{\eps}^{c}]\subseteq\mathrm{ext}\langle A_{\eps}\rangle$
where the latter inclusion follows easily from \cite[Thm.~10.(ii)]{GKV24}
and the definition of the strong exterior.
\end{proof}
\begin{rem}
Let $A\subseteq\mathcal{G}$, $x\in\mathcal{G}$ and consider the
property \ref{enu:ExtPro} of Thm.~\ref{thm:glb=000020implies=000020inf}
in the general context of an $\rti$-module $\mathcal{G}$ clearly
stated as 
\begin{equation}
\forall x\in\mathcal{G}\setminus A\,\exists L\subseteq_{0}I:\,x\in_{L}\mathrm{ext}(A)\label{eq:ExtProp2}
\end{equation}
\begin{enumerate}
\item Clearly, the property holds when $\mathcal{G}$ is a Colombeau space
based on a family of normed spaces $E$, and $A$ is strongly internal
set. The proof is identical to the proof of Prop.~\ref{prop:ExtProp}.
\item The property holds for any ball $B_{r}(y)\subseteq\mathcal{G}$ since
$A:=\{z\in\mathcal{G}\mid||z-y||\geq r\}=\mathrm{ext}(B_{r}(y))$.
Indeed, if $z\in A$ then for any $a\in B_{r}(y)$ we have 
\[
||z-a||\geq||z-y||-||y-a||>0
\]
which proves the direct inclusion. Let now $t\in\mathcal{G}$ be such
that $||t-a||>0$ for all $a\in B_{r}(y)$. For any $q\in\N$, set
$a_{q}:=y+(r-\diff\rho^{q})||t-y||^{-1}(t-y)\in B_{r}(y)$. Then 
\[
||t-y||=||t-a_{q}||+||a_{q}-y||>r-\diff\rho^{q}.
\]
Letting $q$ tends to $\infty$ we get $||t-y||\geq r$, which proves
the opposite inclusion.
\item The property holds also whenever the property $\mathrm{ext}(\mathrm{ext}(A))\subseteq A$
holds. Indeed, assume (by negation) that $x\not\in_{L}\mathrm{ext}(A)$
for all $L\subseteq_{0}I$. \cite[Lem.~3.4]{GiKu18} implies that
$x\in\mathrm{ext}(\mathrm{ext}(A))$. Therefore, if we assume that
$\mathrm{ext}(\mathrm{ext}(A))\subseteq A$, we get that $x\in A$,
which contradicts $x\not\in A$.
\end{enumerate}
\end{rem}

Before defining the Minkowski functional on convex subsets of $\mathcal{G}$,
we need the following proposition.
\begin{prop}
\label{prop:proposition12}Let $C$ be a non-empty sharply open convex
subset of $\mathcal{G}$, and let $x_{0}\in C$. Then, for any $x\in\mathcal{G}$
the set 
\[
C_{x}:=\{\alpha\in\rti_{>0}\mid x_{0}+\alpha^{-1}(x-x_{0})\in C\}
\]
is a non-empty  sharply open convex subset for every $x\in\mathcal{G}$
(note that $C_{x}$ depends also on $x_{0}$). Moreover, if property
\eqref{eq:ExtProp2} holds for the set $C\subseteq\mathcal{G}$, then
property \ref{enu:ExtPro} of Thm.~\eqref{thm:glb=000020implies=000020inf}
holds for the set $C_{x}\subseteq\rti$ as well.
\end{prop}

\begin{proof}
By a translation argument we can assume that $x_{0}=0$. We first
start by proving that $C_{x}$ is non empty. Since $C$ is a sharply
open set, there exists $r\in\rti_{>0}$ such that $B_{r}(0)\subseteq C$.
It is easy to check that 
\[
\forall\alpha\in\rti_{>0}:\ \alpha>\frac{||x||}{r}\Longrightarrow\alpha^{-1}x\in B_{r}(0)\subseteq C
\]
which proves that $C_{x}$ in non-empty. We prove now that $C_{x}$
is sharply open. Let $\lambda\in C_{x}$. By definition $\lambda\in\rti_{>0}$
and $\lambda^{-1}x\in C$. Since $C$ is sharply open, there exists
$s\in\rti_{>0}$ such that $B_{s}(\lambda^{-1}x)\subseteq C$. One
can easily deduce that $\beta\in C_{x}$ for all $\beta>\lambda-\frac{s\lambda^{2}}{||x||+s\lambda}>0$
where the latter inequality holds when choosing $s$ sufficiently
large. We prove now the convexity of $C_{x}$; Let $\alpha$, $\beta\in\rti_{>0}$
be such that $\alpha^{-1}x$, $\beta^{-1}x\in C$. By Lem.~7, \emph{(i)}
of \cite{MTAG21}, we have $\alpha\leq\beta$ or $\beta\leq\alpha$
or there exists $L\subseteq_{0}I$ such that $L^{c}\subseteq_{0}I$,
$\alpha\leq_{L}\beta$ and $\beta\leq_{L^{c}}\alpha$. Hence, without
loss of generality we assume that $\alpha\leq\beta$ in $\rti|_{L}$,
and in the following we work in the ring $\rti|_{L}$. Since $C$
is convex 
\begin{equation}
\forall\lambda\in[0,1]:\,\alpha^{-1}x+\lambda(\beta^{-1}-\alpha^{-1})x\in C.\label{eq:3.26}
\end{equation}
We would like to prove that 
\begin{equation}
\forall t\in[0,1]:\,(\alpha+t(\beta-\alpha))^{-1}x\in C.\label{eq:3.27}
\end{equation}
For all $t\in[0,1]$, we set 
\[
\lambda_{t}=\frac{t\beta}{\alpha+t(\beta-\alpha)}.
\]
which is well defined and belongs to $[0,1]$. Replacing $\lambda_{t}$
in \eqref{eq:3.26} we get \eqref{eq:3.27}.

We prove now the second part. Let $\lambda\in\rti$ be such that $\lambda\not\in C_{x}$.
If $\lambda\leq_{L}0$ for some $L\subseteq_{0}I$ then $\lambda\in_{L}\mathrm{ext}(C_{x})$.
If $\lambda\not\leq_{L}0$ for all $L\subseteq_{0}I$, then $\lambda>0$
(see Lem.~5 of \cite{MTAG21}). Hence, since $\lambda\not\in C_{x}$,
we have $\lambda^{-1}x\not\in C$. Since the set $C$ satisfies \eqref{eq:ExtProp2},
there exists $L\subseteq_{0}I$ such that $\lambda^{-1}x\in_{L}\mathrm{ext}(C)$.
The latter implies in particular that $||x||>_{L}0$ because $x_{0}=0\in C$.
For $\beta\in C_{x}$, we have $\beta\in\rti_{>0}$ and $\beta^{-1}x\in C$,
and hence $||\lambda^{-1}x-\beta^{-1}x||>_{L}0$ which implies that
$|\lambda^{-1}-\beta^{-1}|>_{L}0$ and hence $|\lambda-\beta|>_{L}0$.
Therefore, $\lambda\in_{L}\mathrm{ext}(C_{x})$, which completes the
proof.
\end{proof}
For $A\subseteq\rti$, we write $\exists\text{g.l.b.}(A)\in\rti$
to say that the greatest lower bound of the set $A$ exists in $\rti$.
\begin{defn}
Let $\mathcal{G}$ be an $\rti$-normed $\rti$-module, $C\subseteq\mathcal{G}$
be a sharply open convex subset and $x_{0}\in C$. Assume that the
set $C$ satisfies property \eqref{eq:ExtProp2} and that the following
property holds 
\begin{equation}
\forall x\in\mathcal{G}:\,\exists\text{g.l.b.}\{\alpha\in\rti_{>0}\mid x_{0}+\alpha^{-1}(x-x_{0})\in C\}\in\rti.\label{eq:existence_glb}
\end{equation}
Then, the map $p:\mathcal{G}\longrightarrow\rti_{\geq0}$ given by
\begin{equation}
p(x)=\text{inf}\{\alpha\in\rti_{>0}\mid x+\alpha^{-1}(x-x_{0})\in C\}.\label{eq:Minkowski_functional}
\end{equation}
is well defined and is called \emph{Minkowski functional} of the set
$C$ (associated with the generalized point $x_{0})$.
\end{defn}

The well-posedness of $p$ is a consequence of Prop.~\ref{prop:proposition12}
and of Thm.~\ref{thm:glb=000020implies=000020inf}.
\begin{lem}
\label{lem:Minkowsky=000020prop}Let $C$ be a sharply open convex
set of $\mathcal{G}$ that contains $0$. Assume that the Minkowski
functional $p$ of the set $C$ (associated with the generalized point
$0$) is well defined . Then, the following properties holds
\begin{enumerate}
\item \label{enu:mink1}$p(\lambda x)=\lambda p(x)$ for every $x\in\mathcal{G}$
and for every $\lambda\in\rti_{>0}$.
\item \label{enu:mink2}There exists $M\in\rti_{>0}$ such that $\forall x\in\mathcal{G}:\ 0\le p(x)\leq M||x||$.
\item \label{enu:mink3}$C=\{x\in\mathcal{G}\mid p(x)<1\}$.
\item \label{enu:mink4}$p(x+y)\leq p(x)+p(y)$, for every $x$, $y\in\mathcal{G}$.
\end{enumerate}
\end{lem}

\begin{proof}
\ref{enu:mink1}: Let $x\in\mathcal{G}$ and $\lambda\in\rti_{>0}$
. We have 
\begin{flalign*}
p(\lambda x) & =\text{inf}\{\alpha\in\rti_{>0}\mid(\alpha\lambda^{-1})^{-1}x\in C\}\\
 & \overset{\beta=\alpha\lambda^{-1}}{=}\text{inf}\{\lambda\beta\in\rti_{>0}\mid\beta{}^{-1}x\in C\}\\
 & =\lambda\text{inf}\{\beta\in\rti_{>0}\mid\beta{}^{-1}x\in C\}=\lambda p(x)
\end{flalign*}
where we used the property $\text{inf}(\lambda A)=\lambda\text{inf}(A)$
in the third eq\textcolor{black}{uality, see Lem.~24 }\textcolor{black}{\emph{(i)}}\textcolor{black}{{}
of \cite{MTAG21} for the proof.}

\ref{enu:mink2}: Since $C$ is a sharply open set, there exists $r\in\rti_{>0}$
such that $B_{r}(0)\subseteq C$. Let $x\in\mathcal{G}$. Then $(r-\diff\rho^{q})\frac{x}{||x||+\diff\rho^{q}}\in C$
for $q$ sufficiently large. By definition of $p$ we have 
\begin{equation}
p(x)\leq\frac{||x||+\diff\rho^{q}}{r-\diff\rho^{q}}
\end{equation}
since $q$ is arbitrary large, letting $q$ tends to $\infty$ we
obtain \eqref{enu:mink2} with $M=r^{-1}.$

\ref{enu:mink3}: Let $x\in C$. Since $C$ is a sharply open set
$(1+\diff\rho^{q})x\in C$ for $q$ sufficiently large. Thus, $p(x)\leq\frac{1}{(1+\diff\rho^{q})}<1$.
Conversely, if $p(x)<1$, then take $\alpha:=p(x)+\diff\rho^{q}<1$
for $q$ sufficiently larg\textcolor{black}{e (see Lem.~8 of \cite{GKV24}).
}It follows that $\alpha^{-1}x\in C$. Since $C$ is convex and $0,$
$(\alpha^{-1}x)\in C$, we gave $t(\alpha^{-1}x)+(1-t)0\in C$ for
all $t\in[0,1]$$.$ Choosing $t=\alpha$ we obtain $x\in C$.

\ref{enu:mink4}: Let $x$, $y\in\mathcal{G}$ and let $q\in\mathbb{N}.$
\ref{enu:mink1} and \ref{enu:mink3} imply that $\frac{x}{p(x)+\diff\rho^{q}}$,
$\frac{y}{p(y)+\diff\rho^{q}}\in C$. Indeed, using \ref{enu:mink1}
with $\lambda=(p(x)+\diff\rho^{q})^{-1}$ we obtain $p\left(\frac{x}{p(x)+\diff\rho^{q}}\right)=\frac{p(x)}{p(x)+\diff\rho^{q}}$.
Since the right hand side of the latter equality is strictly less
than $1$, \ref{enu:mink3} implies that $\frac{x}{p(x)+\diff\rho^{q}}\in C$.
Since $C$ is convex, $\frac{tx}{p(x)+\diff\rho^{q}}+\frac{(1-t)y}{p(y)+\diff\rho^{q}}\in C$
for all $t\in[0,1].$ Choosing $t=\frac{p(x)+\diff\rho^{q}}{p(x)+p(y)+2\diff\rho^{q}}\in[0,1]$
we obtain $\frac{x+y}{p(x)+p(y)+2\diff\rho^{q}}\in C$. Using now
\ref{enu:mink1} and \ref{enu:mink3} we obtain 
\[
\frac{p(x+y)}{p(x)+p(y)+2\diff\rho^{q}}=p\left(\frac{x+y}{p(x)+p(y)+2\diff\rho^{q}}\right)<1.
\]
 It follows 
\[
p(x+y)<p(x)+p(y)+2\diff\rho^{q}.
\]
Letting $q\rightarrow\infty$ we obtain $p(x+y)\leq p(x)+p(y)$ which
completes the proof.
\end{proof}
\begin{example}
\label{exa:example=000020minkowski=000020}Let $C$ be a sharply open
convex subset of an $\rti$-normed $\rti$-module $\mathcal{G}$ containing
$0$, and assume that the Minkowski functional $p$ of the set $C$
(with $x_{0}=0$) is well defined.
\begin{enumerate}
\item If $C=\mathcal{G}$, then $p$ is identically null. Indeed, for every
$x\in\mathcal{G}$ we have $\lambda^{-1}x\in\mathcal{G}$ for every
$\lambda\in\rti_{>0}$. This yields $\{\alpha\in\rti_{>0}\mid\alpha^{-1}x\in C\}=\rti_{>0}$.
Thus, $\inf\rti_{>0}=\inf\{\alpha\in\rti_{>0}\mid\alpha^{-1}x\in C\}=0$.
\item \label{enu:MinkBall}If $C=B_{r}(0)$ with $r\in\rti_{>0}$ then $p(x)=\frac{||x||}{r}$.
Indeed, for every $x\in\mathcal{G}$ \textcolor{black}{we have}\textcolor{red}{{}
}\textcolor{black}{
\[
\{\alpha\in\rti_{>0}\mid\alpha^{-1}x\in B_{r}(0)\}=\left[\alpha\in\rti_{>0}\mid\frac{||x||}{r}<\alpha\right]
\]
which implies easily the desired result.}
\item \label{enu:minkowsku=000020sharply=000020bounded=000020}If $C$ is
sharply bounded, then $p(x)=0$ if and only if $x=0$. Indeed, when
$x=0$ then $\alpha^{-1}x=0$ for every $\alpha\in\rti_{>0}$. Thus,
$p(x)=0$. Conversely, since $C$ is sharply bounded, there exists
$q\in\mathbb{N}$ such that $C\subseteq B_{\diff\rho^{-q}}(0)$. Hence
\[
\{\alpha\in\rti_{>0}\mid\alpha^{-1}x\in C\}\subseteq\{\alpha\in\rti_{>0}\mid\alpha^{-1}x\in B_{\diff\rho^{-q}}(0)\}.
\]
By taking the infimum and using\textcolor{black}{{} }\ref{enu:MinkBall}\textcolor{black}{{}
and Lem.~24 of \cite{MTAG21} we ob}tain $\diff\rho^{q}||x||\leq p(x).$
Thus, $p(x)=0$ implies $x=0$.
\end{enumerate}
\end{example}

In the next proposition we give the first example of subsets for which
the associated Minkowski functional is well defined.
\begin{prop}
\label{prop:=000020minkowski=000020strongly=000020internal=000020}Let
$C\subseteq\mathcal{G}_{E}$ be a non-empty convex strongly internal
set. Assume that $C$ has a sharply bounded representative $(C_{\eps})$
that satisfies the following property: 
\begin{equation}
\forall(a_{\eps})\in\prod_{\eps\in I}C_{\eps}\,\forall q\in\N\,\exists u_{q}\in C:\,||u_{q}-[a_{\eps}]||\leq\diff\rho^{q}.\label{eq:1.2-1}
\end{equation}
Then, for any $x_{0}\in C$, the Minkowski functional $p$ of the
set $C$ associated with the generalized point $x_{0}$ is well defined
and 
\begin{equation}
\forall[x_{\eps}]\in\mathcal{G}_{E}:\,p([x_{\eps}])=[\text{\ensuremath{\inf}}\{\lambda\in\mathbb{R}_{>0}\mid x_{0\varepsilon}+\lambda^{-1}(x_{\varepsilon}-x_{0\varepsilon})\in C_{\eps}\}].\label{eq:g.l.b_internal}
\end{equation}
 
\end{prop}

\begin{rem}
\textcolor{black}{~}
\begin{enumerate}
\item \textcolor{black}{Note that }\eqref{eq:g.l.b_internal} does not hold
for any representative $(C_{\eps})$ of $C$. Take for example $C=(-1,1)$
and let $(C_{\eps})$ be a representative of $C$ defined by $\forall^{0}\eps:\,C_{\eps}:=(-1,1)\cup(2-z_{\eps},2+z_{\eps})$
where $(z_{\eps})\in\mathcal{N}_{\R^{+*}}$. For $x=2$, one can check
that 
\[
\forall\eps\in I:\,\sigma_{\eps}(2):=\inf\{\lambda\in\R_{>0}\mid\lambda^{-1}2\in C_{\eps}\}=\frac{2}{2+z_{\eps}},
\]
whereas $\inf\{\lambda\in\rti_{>0}\mid\lambda^{-1}x\in C\}=2$. Hence,
\eqref{eq:g.l.b_internal} does not hold, but clearly it holds with
the representative $(\tilde{C}_{\eps})$ given by $\forall\eps:\,\tilde{C}_{\eps}:=(-1,1)$. 
\item If $C=\langle C_{\eps}\rangle\subseteq\rti^{2}$ be such that 
\[
\forall\eps:C_{\eps}:=\left((-1,1)\times(-1,1)\right)\cup\left((-1,2)\times(-z_{\eps},z_{\eps})\right)
\]
where $(z_{\eps})$ is non-negative negligible net. One can check
that 
\[
\sigma_{\eps}((0.5,0))=\inf\{\lambda\in\R_{>0}\mid\lambda^{-1}(0.5,0)\in C_{\eps}\}=0.25
\]
whereas 
\[
\sigma_{\eps}((0.5,2z_{\eps}))=\inf\{\lambda\in\R_{>0}\mid\lambda^{-1}(0.5,2z_{\eps})\in C_{\eps}\}=0.5.
\]
Thus, for this representative $(C_{\eps})$ of $C$, the right hand
side of \eqref{eq:g.l.b_internal} depends on the representatives
of $x$.\\
Moreover, one can easily see that property \eqref{eq:1.2-1} does
not hold e.g.~with the net $(a_{\eps})$ given by $a_{\eps}:=(0,3/4)$
for all $\eps$. 
\item Let $C=\langle C_{\eps}\rangle\subseteq\mathcal{G}_{E}$ be a non-empty
subset. Then, $\overline{C}\subseteq[C_{\eps}]$. Indeed, let $x=[x_{\eps}]\in\overline{C}$
and let $(x_{n})$ be a sequence $x_{n}=[x_{n\eps}]\in C$ such that
\[
\forall n\in\N\,\exists\eps_{n}\in I\,\forall\eps\in(0,\eps_{n}]:x_{n\eps}\in C_{\eps}\,\,\,\text{and}\,\,\,||x_{n\eps}-x_{\eps}||\leq\rho_{\eps}^{n}.
\]
Without loos of generality, we assume that the sequence $(\eps_{n})$
is non-increasing and converging to $0$. Hence, setting $\tilde{x}_{\eps}:=x_{n\eps}$
whenever $\eps\in(\eps_{n+1},\eps_{n}]$ we clearly have $[\tilde{x}_{\eps}]=x\in[C_{\eps}]$,
and hence the inclusion is proved. On the other hand, assuming that
the assumption \eqref{eq:1.2-1} holds, the inclusion $[C_{\eps}]\subseteq\overline{C}$
easily follows. 
\end{enumerate}
\end{rem}

\begin{proof}[Proof of Prop.~\ref{prop:=000020minkowski=000020strongly=000020internal=000020}.]
 By a translation argument we can take $x_{0}=0$. Set
\begin{equation}
\forall\eps\in I\,\forall x\in E:\,\sigma_{\eps}(x):=\inf\{\lambda\in\mathbb{R}_{>0}\mid\lambda^{-1}x\in C_{\varepsilon}\}.\label{eq:sigma_eps}
\end{equation}
First, note that since $C$ is a sharply open set and the representative
$(C_{\eps})$ is sharply bounded, there exists $q_{0}\in\N$ such
that 
\[
\forall^{0}\eps:\,B_{\rho_{\eps}^{q_{0}}}(0)\subseteq C_{\eps}\subseteq B_{\rho_{\eps}^{-q_{0}}}(0).
\]
It follows that 
\begin{equation}
\forall^{0}\eps\,\forall x\in E_{\eps}:\,\rho_{\eps}^{q_{0}}||x||_{\eps}\leq\sigma_{\eps}(x)\leq\rho_{\eps}^{-q_{0}}||x||_{\eps}.\label{eq:sigma_eps2}
\end{equation}
Consider now $[x_{\eps}]\in\mathcal{G}_{E}$. Then, $(\sigma_{\eps}(x_{\eps}))\in\R_{\rho}$
(by \eqref{eq:sigma_eps2}). We now claim that 
\begin{equation}
\forall q\in\N:([\sigma_{\eps}(x_{\eps})]+\diff\rho^{q})^{-1}x\in C.\label{eq:2.50}
\end{equation}
Indeed, set
\[
S_{0}:=\{\eps\in I\mid x_{\eps}=0\},\quad S_{*}:=\{\eps\in I\mid x_{\eps}\neq0\}.
\]
Assume that $S_{0}\subseteq_{0}I$. Then, $[x_{\eps}]=_{S}0$, $[\sigma_{\eps}(x_{\eps})]=_{S}0$
(by \eqref{eq:sigma_eps2}), and 
\begin{equation}
\forall q\in\N:([\sigma_{\eps}(x_{\eps})]+\diff\rho^{q})^{-1}x=_{S}0\in C.\label{eq:Inclusion}
\end{equation}
Assume now that $S_{*}\subseteq_{0}I$. For simplicity of notation
we assume that $S_{*}=I$. Let $(\alpha_{\eps})$ be a non-negative
negligible net. By the definition of the net $(\sigma_{\eps}(x_{\eps}))$,
there exists a positive negligible net $(\beta_{\eps})$ satisfying
\[
\forall\eps:\,\beta_{\eps}\leq\alpha_{\eps}||x_{\eps}||_{\eps}\,\,\,\text{and}\,\,\,z_{\eps}:=(\sigma_{\eps}(x_{\eps})+\beta_{\eps})^{-1}x_{\eps}\in C_{\eps}
\]
Recall from \eqref{eq:sigma_eps2} that $\sigma_{\eps}(x_{\eps})>0$
for all $\eps$, and hence $z_{\eps}$ is well defined. By \eqref{eq:1.2-1}
we have that 
\[
\forall q\in\N\,\exists u_{q}\in C:\,||z-u_{p}||<\diff\rho^{p}
\]
where $z=[z_{\eps}]\in\mathcal{G}_{E}$ (since the representative
$(C_{\eps})$ is sharply bounded). Note that $||z-u_{p}||>0$ for
all $p\in\N$. Indeed, assume that $||z-u_{p}||=_{L}0$ for some $p\in\N$
and for some $L\subseteq_{0}I$, and hence $z\in_{L}C$. Since $C$
is a sharply open set, there exits $r\in\N$ such that $\forall\eps\in L:\,B_{\rho_{\eps}^{r}}(z_{\eps})\subseteq C_{\eps}$,
which implies that $((\sigma_{\eps}(x_{\eps})+\beta_{\eps})^{-1}+\rho_{\eps}^{r+1}||x_{\eps}||_{\eps}^{-1})x_{\eps}\in C_{\eps}$
and hence 
\[
\forall\eps\in L:\,\sigma_{\eps}(x_{\eps})\leq((\sigma_{\eps}(x_{\eps})+\beta_{\eps})^{-1}+\rho_{\eps}^{r+1}||x_{\eps}||_{\eps}^{-1})^{-1}.
\]
By \eqref{eq:sigma_eps2} we have that $\sigma_{\eps}(x_{\eps})>0$
for all $\eps\in L$ and 
\[
\forall\eps\in L:\,\alpha_{\eps}\rho_{\eps}^{-2q_{0}}||x_{\eps}||_{\eps}^{-1}\geq\beta_{\eps}\rho_{\eps}^{-2q_{0}}||x_{\eps}||_{\eps}^{-2}\geq\frac{\beta_{\eps}}{\sigma_{\eps}(x_{\eps})(\sigma_{\eps}(x_{\eps})+\beta_{\eps})}\geq\rho_{\eps}^{r+1}||x_{\eps}||_{\eps}^{-1},
\]
which is not true since the net $(\alpha_{\eps})$ is negligible.
Therefore, $||z-u_{p}||>0$ for every $p\in\N$. Since $0\in C$ and
$C$ is a strongly internal set, there exists $m$ sufficiently large
such that 
\[
\forall p\in\N:\,\frac{\diff\rho^{m}}{||z-u_{p}||}(z-u_{p})\in C.
\]
By convexity of $C$ we have 
\[
\forall p\in\N\,\forall\alpha\in[0,1]:\frac{\diff\rho^{m}}{||z-u_{p}||}(z-u_{p})+\alpha\left(u_{p}-\frac{\diff\rho^{m}}{||z-u_{p}||}(z-u_{p})\right)\in C
\]
Choosing $\alpha=\frac{\diff\rho^{m}}{||z-u_{p}||+\diff\rho^{m}}\in[0,1]$
to get 
\begin{equation}
\frac{\diff\rho^{m}}{||z-u_{p}||+\diff\rho^{m}}z\in C.\label{eq:2.54}
\end{equation}
For any $q\in\N$, we choose $p$ sufficiently large to get 
\begin{equation}
\left[\frac{\sigma_{\eps}(x_{\eps})+\beta_{\eps}}{\sigma_{\eps}(x_{\eps})+\rho_{\eps}^{q}}\right]\leq\frac{\diff\rho^{m}}{||z-u_{p}||+\diff\rho^{m}}.\label{eq:2.55}
\end{equation}
From $0\in C$, \eqref{eq:2.54}, \eqref{eq:2.55} and by convexity
of $C$ we obtain 
\[
\forall q\in\N:\,\left[\frac{\sigma_{\eps}(x_{\eps})+\beta_{\eps}}{\sigma_{\eps}(x_{\eps})+\rho_{\eps}^{q}}\right]z=\left[(\sigma_{\eps}(x_{\eps})+\rho_{\eps}^{q})^{-1}x_{\eps}\right]=([\sigma_{\eps}(x_{\eps})]+\diff\rho^{q})^{-1}x\in C,
\]
which, together with \eqref{eq:Inclusion}, implies \eqref{eq:2.50}.
Therefore, our claim is proved. We prove now that 
\begin{equation}
[x_{\eps}]=[x'_{\eps}]\Longrightarrow[\sigma_{\eps}(x_{\eps})]=[\sigma_{\eps}(x'_{\eps})]=:\sigma(x),\label{eq:infRepresentative}
\end{equation}
which shows that the right hand side of \eqref{eq:g.l.b_internal}
doesn't depend on the representative of $x$. Indeed, assume that
there exist $n\in\N$ and $L\subseteq_{0}I$ such that 
\begin{equation}
\forall\eps\in L:\,\sigma_{\eps}(x'_{\eps})+\rho_{\eps}^{n}\leq\sigma_{\eps}(x{}_{\eps}).\label{eq:2.51}
\end{equation}
By \eqref{eq:2.50} we have $[\sigma_{\eps}(x'_{\eps})+\rho_{\eps}^{n+1}]^{-1}x\in C$.
Since $C$ is a strongly internal set we have 
\[
\forall^{0}\eps:\,(\sigma_{\eps}(x'_{\eps})+\rho_{\eps}^{n+1})^{-1}x_{\eps}\in C_{\eps}
\]
which contradicts \eqref{eq:2.51}. Therefore, the right hand side
of \eqref{eq:g.l.b_internal} does not depend on the representative
of $x$. Finally, \eqref{eq:g.l.b_internal} is a consequence of \eqref{eq:sigma_eps}
and \eqref{eq:2.50}.
\end{proof}
\begin{rem}
The left hand side of \eqref{eq:g.l.b_internal} defines the Minkowski
function of the set $C$ (which is an $\rti$-sublinear functional
as shown in Lem.~\ref{lem:Minkowsky=000020prop}). However, the map
$E_{\eps}\ra\R_{\geq0}$, $x\ra\sigma_{\eps}(x)$ is not an $\R$-sublinear
map since $C_{\eps}$ is not necessary convex. Therefore, assuming
that the representative $(C_{\eps})$ is consisting of convex sets
we get that the Minkowski functional is generated by a family of $\R$-sublinear
maps $(\sigma_{\eps})$. 
\end{rem}

In the next proposition, we prove that property \eqref{eq:1.2-1}
holds if $C_{\eps}$ is convex for all $\eps$ small. 
\begin{prop}
\label{prop:ConvImlies}Let $C\subseteq\mathcal{G}_{E}$ be a non-empty
strongly internal set. Assume that $C$ has a sharply bounded representative
$(C_{\eps})$ consisting of convex subsets. Then, \eqref{eq:1.2-1}
holds. 
\end{prop}

\begin{proof}
Since $C$ is a non-empty set, we assume for simplicity that $0\in C$.
Let $(C_{\eps})$ be a sharply bounded representative of $C$ made
of convex sets, and let $(a_{\eps})\in\prod_{\eps\in I}C_{\eps}$.
Since $C$ is a sharply open set, there exists $m\in\N$ such that
$B_{\rho_{\eps}^{m}}(0)\subseteq C_{\eps}$ for all $\eps$ small.
Setting 
\[
\forall\eps\,\forall q\in\N:\,a_{\eps q}:=\left(1-\frac{\rho_{\eps}^{2q}}{||a_{\eps}||_{\eps}+\rho_{\eps}^{q}}\right)a_{\eps}.
\]
We claim now that $a_{q}:=[a_{\eps q}]\in C$. First, since the representative
$(C_{\eps})$ is sharply bounded, there exists $r\in\N$ such that
$||a_{\eps q}||_{\eps}\leq||a_{\eps}||_{\eps}\leq\rho_{\eps}^{-r}$
for all $q$ and for all $\eps$ small, which implies in particular
that $(a_{\eps q})\in\mathcal{M}_{E}$. Moreover, we clearly have
\[
\forall q\in\N\,\exists s=m+2q+2r\,\forall^{0}\eps\,\forall z\in B_{\rho_{\eps}^{s}}(0):\,\rho_{\eps}^{-2q}(||a_{\eps}||_{\eps}+\rho_{\eps}^{q})z\in B_{\rho_{\eps}^{m}}(0)\subseteq C_{\eps}.
\]
By convexity of $C_{\eps}$ we have 
\[
\forall q\in\N\,\forall^{0}\eps\,\forall z\in B_{\rho_{\eps}^{s}}(0)\,\forall\lambda\in[0,1]:\,\lambda\frac{||a_{\eps}||_{\eps}+\rho_{\eps}^{q}}{\rho_{\eps}^{2q}}z+(1-\lambda)a_{\eps}\in C_{\eps}.
\]
Choosing $\lambda:=\frac{\rho_{\eps}^{2q}}{||a_{\eps}||_{\eps}+\rho_{\eps}^{q}}\in[0,1]$
we get that $a_{\eps q}+z\in C_{\eps}$ for all $\eps$ small. Hence,
$B_{\rho_{\eps}^{s}}(a_{\eps q})\subseteq C_{\eps}$ for all $\eps$
small. Therefore, $a_{q}\in C$. Moreover, we clearly have $||a_{q}-a||\leq\diff\rho^{q}$
which proves \eqref{eq:1.2-1}. 
\end{proof}
We combine now Prop.~\ref{prop:=000020minkowski=000020strongly=000020internal=000020}
and Prop.~\ref{prop:ConvImlies} to get the following result 
\begin{cor}
\label{cor:4.10=000020ImpMink}Let $C\subseteq\mathcal{G}_{E}$ be
a non-empty strongly internal subset. Assume that the set $C$ has
a sharply bounded representative consisting of convex sets. Then,
the Minkowski functional $p$ of the set $C$ is well defined and
is generated by a net of sublinear maps. 
\end{cor}

We present now an example of a subset of $\rti$ that does not satisfy
property \eqref{eq:existence_glb}.
\begin{example}
The set $D_{\infty}\subseteq\rti$ of all infinitesimal of $\rti$
does not satisfy property \eqref{eq:existence_glb} e.g.~when $x$
is invertible and $x_{0}=0$. Indeed, assume by contradiction that
\begin{equation}
\exists x\in\rti_{>0}:\,\mathrm{\exists g.l.b.}\{\alpha\in\rti_{>0}\mid\alpha^{-1}x\in D_{\infty}\}.\label{eq:2.61}
\end{equation}
Clearly, the set $D_{\infty}$ is convex non-empty, sharply open,
and satisfies \ref{enu:ExtPro} of Thm.~\ref{thm:glb=000020implies=000020inf}.
Thus, by Thm.~\ref{thm:glb=000020implies=000020inf}, there exists
$p(x)\in\rti$ satisfying 
\begin{equation}
\forall\lambda\in\rti_{>0}\text{ with }\lambda^{-1}x\in D_{\infty}:\,0\leq p(x)\leq\lambda\label{eq:3.39}
\end{equation}
\begin{equation}
\forall q\in\mathbb{N}\,\exists\lambda_{q}\in\rti_{>0}\text{ with }\lambda_{q}^{-1}x\in D_{\infty}\text{ and }\lambda_{q}\leq p(x)+\diff\rho^{q}.\label{eq:3.40}
\end{equation}
Using \eqref{eq:3.39} with $\lambda_{q}/2$ and \eqref{eq:3.40}
we get 
\begin{equation}
\forall q\in\N:\,0\leq2p(x)\leq\lambda_{q}\leq p(x)+\diff\rho^{q},
\end{equation}
which implies that $p(x)=0$. On the other hand, since $x$ is invertible,
there exists $s>0$ such that $\diff\rho^{-s}x\geq1$ which leads
to a contradiction. 
\end{example}

The following theorem is the first geometric form of the Hahn-Banach
theorem in $\mathcal{G}_{E}$. 
\begin{thm}
\label{thm:geo=000020form=0000201}Let $A=\langle A_{\eps}\rangle$,
$B=[B_{\eps}]\subseteq\mathcal{G}_{E}$ be two non-empty and strongly
disjoint sets. Assume that the nets $(A_{\eps})$, $(B_{\eps})$ are
sharply bounded and consisting of convex sets. Then, there exists
a continuous $\rti$-linear map $f=[f_{\eps}]\ra\rti$ non-identically
null and $\alpha\in\rti$ such that \textcolor{black}{
\begin{equation}
\forall y\in A\,\forall z\in B:\,f(y)\leq\alpha\leq f(z).\label{eq:3.51}
\end{equation}
i.e.~the closed hyperplane $\{f=\alpha\}$ separates $A$ and $B$.
}Moreover, if $B$ is a cone,~i.e. a set that satisfies $\lambda x\in B$
for all $x\in B,\lambda\in\rti_{\geq0}$, then we can choose $\alpha=0$. 
\end{thm}

We first prove the following lemma
\begin{lem}
\textcolor{black}{\label{lem:separation=000020a=000020point}Let $C=\langle C_{\eps}\rangle\subseteq\mathcal{G}_{E}$
be a non-empty set and let $x_{0}=[x_{0\eps}]\in\mathcal{G}_{E}$
such that $\left\{ x_{0}\right\} $ and $C$ are strongly disjoint.
}Assume that $C_{\eps}$ is convex for all $\eps$ small, and that
the representative $(C_{\eps})$ is sharply bounded. Then, there exists
a continuous $\rti$-linear map $f=[f_{\eps}]\ra\rti$ non-identically
null and $\alpha\in\rti$ such that \textcolor{black}{
\[
\forall x\in C:\,f(x)\leq\alpha=f(x_{0}).
\]
i.e.~the closed hyperplane $\{f=\alpha\}$ separates $C$ and $\left\{ x_{0}\right\} $.}
\end{lem}

\begin{proof}
\textcolor{black}{By a translation argument we can assume that $0\in C$.
Since $\left\{ x_{0}\right\} $ and $C$ are strongly disjoint, we
have $||x_{0}||>0$. Consider the internal $\rti$-submodule of $\mathcal{G}_{E}$
given by $\mathcal{F}:=\rti x_{0}=[\R x_{0\eps}]$, where $[x_{0\eps}]=x_{0}$.
On $\mathcal{F}$ we define the $\rti$-linear map $g=[g_{\varepsilon}]:\mathcal{F}\longrightarrow\rti$
by 
\[
\forall t=[t_{\eps}]\in\rti:\,g(tx_{0}):=[g_{\eps}(t_{\eps}x_{0\eps})]:=[t_{\eps}]=t.
\]
This means that $g(x_{0})=[g_{\varepsilon}(x_{0\varepsilon})]=1$.
By Cor.~}\ref{cor:4.10=000020ImpMink},\textcolor{black}{{} }the Minkowski
functional $p$ of the set $C$ (associated with $0$) is well defined
and generated by a net $(p_{\eps})$ of $\R$-sublinear maps. \textcolor{black}{We
claim that 
\begin{equation}
\forall x\in\mathcal{F}:\,g(x)\leq p(x).\label{eq:3.47}
\end{equation}
Indeed, since $x_{0}$ and $C$ are strongly disjoint, we have $x\not\in_{L}C$
for all $L\subseteq_{0}I$ which, by Lem.~\ref{lem:Minkowsky=000020prop},
\ref{enu:mink3}, implies that $p(x_{0})\not<_{L}1$ for all $L\subseteq_{0}I$
and hence $p(x_{0})\geq1$. Let now $t\in\rti$. By Lem.~7 of \cite{MTAG21},
we have $t\leq0$ or $t\geq0$ or there exists $L\subseteq_{0}I$
such that $L^{c}\subseteq_{0}I$ and $t\leq_{L}0$ and $t\geq_{L^{c}}0$.
In the case $t\geq0$, there exists a representative $(t_{\eps})$
of $t$ consisting of non-negative representative. It follows that
\[
p(tx_{0})=[p_{\eps}(t_{\eps}x_{0\eps})]=[t_{\eps}p_{\eps}(x_{0\eps})]=tp(x_{0})\ge t=g(tx_{0})
\]
where we used in particular the fact that $[p_{\eps}(t_{\eps}x_{0\eps})]$
does not depend on the representative of $t$, and that $p_{\eps}$
is a sublinear map for all $\eps$. Assume now that $t\leq0$. Then
$p(tx_{0})\geq0\geq t=g(tx_{0})$. The case $t\leq_{L}0$ and $t\geq_{L^{c}}0$
for some $L\subseteq_{0}I$ such that $L^{c}\subseteq_{0}I$ follows
similarly to the last two cases. Thus \eqref{eq:3.47} holds. Thanks
to Thm.~\ref{thm:HBT}, there exists an $\rti$-linear map $f=[f_{\eps}]$
on $\mathcal{G}_{E}$ that extends $g$ and satisfies 
\begin{equation}
\forall x\in\mathcal{G}_{E}:\,f(x)\leq p(x).\label{eq:GeoForHBT1}
\end{equation}
By Lem.~\ref{lem:Minkowsky=000020prop} \ref{enu:mink2}, there exists
$M\in\rti_{>0}$ such that $f(x)\leq p(x)\leq M||x||$ for all $x\in\mathcal{G}_{E}$,
which implies that $f$ is continuous. Moreover, it is non-null because
$f(x_{0})=g(x_{0})=1$. Thus, }by Prop.~\ref{hyperplane} \ref{enu:hyperplane1}\textcolor{black}{,
the hyperplane $\tilde{H}:=\{f=1\}$ is closed. Finally, we use Lem.~\ref{lem:Minkowsky=000020prop}
\ref{enu:mink3} and }\eqref{eq:GeoForHBT1}\textcolor{black}{{} to
conclude.}
\end{proof}
We also need the following lemma
\begin{lem}
\label{lem:Sum}Let $(A_{\eps})$, $(B_{\eps})$ be two nets of $\prod_{\eps\in I}E_{\eps}$
such that the sets $[A_{\eps}]$, $[B_{\eps}]$ and $\langle B_{\eps}\rangle$
are non-empty.
\begin{enumerate}
\item \label{enu:Sum1}Assume that at least one of the nets $(A_{\eps})$,
$(B_{\eps})$ is sharply bounded. Then 
\[
[A_{\eps}]+[B_{\eps}]=[A_{\varepsilon}+B_{\eps}]
\]
\item \label{enu:Sum2}Assume that the nets $(A_{\eps})$, $(B_{\eps})$
are sharply bounded and convex for all $\eps$ small. Then 
\[
[A_{\eps}]+\langle B_{\eps}\rangle:=\langle A_{\eps}+B_{\eps}\rangle
\]
\end{enumerate}
\end{lem}

\begin{proof}
The assertion \ref{enu:Sum1} is trivial and the sharp boundedness
assumption is needed only for the opposite inclusion $\supseteq$.
We prove now \ref{enu:Sum2}. Indeed, the direct inclusion is trivial,
and we prove only the inclusion $\langle A_{\eps}+B_{\eps}\rangle\subseteq[A_{\eps}]+\langle B_{\eps}\rangle$.
Let $x=[x_{\eps}]\in\langle A_{\eps}+B_{\eps}\rangle$. Then, by Thm.~10,
\emph{(ii)} of \cite{GKV24} 
\begin{equation}
\exists s\in\N\,\forall^{0}\eps:\,B_{\rho_{\eps}^{s}}(x_{\eps})\subseteq A_{\eps}+B_{\eps}\label{eq:4.44}
\end{equation}
We have in particular $x_{\eps}=y_{\eps}+z_{\eps}$ where $y_{\eps}\in A_{\eps}$
and $z_{\eps}\in B_{\eps}$. Since $B$ is a non-empty and strongly
internal set, there exists $a=[a_{\eps}]\in B$ and $m\in\N$ such
that $B_{\rho_{\eps}^{m}}(a_{\eps})\subseteq B_{\eps}$ for all $\eps$
small. Hence, up to choosing $s$ sufficiently large, the set $\{b\in B_{\eps}\mid d(b,B_{\eps}^{c})\geq\rho_{\eps}^{m+1}\}$
is non-empty for all $\eps$ small. Set $S:=\{\eps\in I\mid d(z_{\eps},B_{\eps}^{c})\geq\rho_{\eps}^{m+1}\}$.
In case $S^{c}\not\subseteq_{0}I$, we have that $z=[z_{\eps}]\in B$
which proves that $x=y+z$, with $y:=[y_{\eps}]\in A$ and $z\in B$
which proves the required inclusion. Assume now that $S^{c}\subseteq_{0}I$.
Note that $||z_{\eps}-a_{\eps}||_{\eps}>0$ for all $\eps\in S^{c}$.
From \eqref{eq:4.44} we have that 
\[
\forall\eps\in S^{c}:\,\tilde{x}_{\eps}:=x_{\eps}+\frac{\rho_{\eps}^{s+1}}{||a_{\eps}-z_{\eps}||_{\eps}}(z_{\eps}-a_{\eps})\in A_{\eps}+B_{\eps}.
\]
There exists then for all $\eps$ small, $\tilde{y}_{\eps}\in A_{\eps}$,
$w_{\eps}\in B_{\eps}$ such that $\tilde{x}_{\eps}:=\tilde{y}_{\eps}+w_{\eps}$
for all $\eps$ small. Since $A_{\eps}$ is convex, 
\[
\forall\eps\in S^{c}\,\forall\lambda\in[0,1]:\,y_{\lambda\eps}:=y_{\eps}+\lambda(\tilde{y}_{\eps}-y_{\eps})\in A_{\eps}.
\]
For all $\eps\in S^{c}$ we have 
\begin{alignat*}{1}
x_{\eps}-y_{\lambda\eps} & =z_{\eps}-\lambda\left(\frac{\rho_{\eps}^{s+1}}{||a_{\eps}-z_{\eps}||_{\eps}}(z_{\eps}-a_{\eps})-w_{\eps}+z_{\eps}\right)\\
 & =a_{\eps}+\lambda(w_{\eps}-a_{\eps})+\left(1-\lambda\frac{\rho_{\eps}^{s+1}+||a_{\eps}-z_{\eps}||_{\eps}}{||a_{\eps}-z_{\eps}||_{\eps}}\right)(z_{\eps}-a_{\eps}).
\end{alignat*}
Choosing $\lambda=\frac{||a_{\eps}-z_{\eps}||_{\eps}}{||a_{\eps}-z_{\eps}||_{\eps}+\rho_{\eps}^{s+1}}=:\lambda_{\eps}\in[0,1]$
we obtain 
\[
\forall\eps\in S^{c}:\,x_{\eps}-y_{\lambda_{\eps}\eps}=a_{\eps}+\frac{||a_{\eps}-z_{\eps}||_{\eps}}{||a_{\eps}-z_{\eps}||_{\eps}+\rho_{\eps}^{s+1}}(w_{\eps}-a_{\eps})=:\tilde{z}_{\eps}.
\]
One can easily show that $[\lambda_{\eps}]|_{S^{c}}\in[0,1)|_{S^{c}}$.
Moreover, the net $(w_{\eps})$ is sharply bounded (since $(B_{\eps})$
is sharply bounded). Hence, one can show, as in the proof of Prop.~\ref{prop:ConvImlies},
that $[\tilde{z}_{\eps}]_{S^{c}}\in B|_{S^{c}}$. On the other hand
$[y_{\lambda_{\eps}\eps}]|_{S^{c}}\in A|_{S^{c}}$ since the net $(A_{\eps})$
is sharply bounded. Therefore, $x=y+z$, with $y:=e_{S}[y_{\eps}]+e_{S^{c}}[y_{\lambda_{\eps}\eps}]\in A$
and $z=e_{S}[z_{\eps}]+e_{S^{c}}[\tilde{z}_{\eps}]\in B$, which completes
the proof.
\end{proof}
\begin{proof}[Proof of Thm.~\ref{thm:geo=000020form=0000201}]
\textcolor{black}{{} Setting $C:=\langle A_{\varepsilon}-B_{\varepsilon}\rangle$.
By assumptions, $C$ is non-empty and has a sharply bounded representative
consisting of convex sets. Moreover, by Lem.~}\ref{lem:Sum}\textcolor{black}{,
we have that $C=A-B$, which implies that that $C$ and $\left\{ 0\right\} $are
strongly disjoint. Hence, by Lem.~\ref{lem:separation=000020a=000020point},
there exists a }continuous\textcolor{black}{{} $\rti$-}linear\textcolor{black}{{}
}map\textcolor{black}{{} $f=[f_{\eps}]:\mathcal{G}_{E}\longrightarrow\rti$}
non-identically null \textcolor{black}{that separates $C$ and $\left\{ 0\right\} $.
More precisely, we have $f(x)\leq0$ for every $x\in C$ which implies
that 
\begin{equation}
\forall y\in A\,\forall z\in B:f(y)\leq f(z).\label{eq:3.50}
\end{equation}
One can easily show that $f\left([B_{\varepsilon}]\right)=[f_{\varepsilon}(B_{\varepsilon})]$.
Indeed, the inclusion $\subseteq$ holds in general, whereas the other
one holds provided the representative $(B_{\eps})$ is sharply bounded.
Thus, one can easily show, using the fact that the net $(B_{\eps})$
is sharply bounded and $f_{\eps}$ is a continuous linear map that
the $\inf f(B)$ exists and it is equal to $[\inf(f_{\eps}(B_{\eps}))]=:\alpha$.
We prove now }\eqref{eq:3.51}.\textcolor{black}{{} Indeed, if we assume
that 
\[
\exists\overline{y}\in A\,\exists q\in\N\,\exists L\subseteq_{0}I:\,\alpha+\diff\rho^{q}<_{L}f(\overline{y}).
\]
By definition of the infimum, there exists $\overline{z}\in B$ such
that $f(\overline{z})\leq\alpha+\diff\rho^{q+1}$. It follows that
\[
f(\overline{z})\leq\alpha+\diff\rho^{q+1}<\alpha+\diff\rho^{q}\leq_{L}f(\overline{y})
\]
which contradicts }\eqref{eq:3.50}\textcolor{black}{. Therefore,
our claim holds and hence the first part of the theorem is proved.}

Assume now that $B$ is a cone. Let $z\in B$ and $\lambda\in\rti_{>0}$.
We have $f(\lambda z)=\lambda f(z)\geq\alpha$ and hence $f(z)\geq\lambda^{-1}\alpha$.
Since the latter holds for all $z\in B$ and for all $\lambda\in\rti_{>0}$
the proof is completed.
\end{proof}
\begin{cor}
Let $A=\langle A_{\varepsilon}\rangle$ and $B=[B_{1\varepsilon}]\cup[B_{2\varepsilon}]$
be two non-empty subsets of $\mathcal{G}_{E}$, and let $\tilde{H}=\{f=\alpha\}$
be a hyperplane. Then, $\tilde{H}$ separates $A$ and $B$ if and
only if it separates $A$ and $B':=[(B_{1\varepsilon}\cup B_{2\varepsilon})_{\varepsilon}]$.
\end{cor}

\begin{proof}
The condition is sufficient since $B\subseteq B'$. Assume now that
the hyperplane $\{f=\alpha\}$ separates $A$ and $B$, that is 
\begin{equation}
\forall x\in A\,\forall y\in B:\,f(x)\leq\alpha\leq f(y).
\end{equation}
Let $z\in B'$ and let $(z_{\eps})$ be a representative of $z$ that
satisfies $\forall^{0}\varepsilon:\,z_{\varepsilon}\in B_{1\varepsilon}\cup B_{2\varepsilon}$.
Set $L_{1}=\{\varepsilon\in I:\,z_{\varepsilon}\in B_{1\varepsilon}\}$,
$L_{2}=\{\varepsilon\in I:z_{\varepsilon}\in B_{2\varepsilon}\}$,
and let $x_{1}\in[B_{1\varepsilon}]$, $x_{2}\in[B_{2\varepsilon}]$.
Set $y_{1}=e_{L_{1}}z+e_{L_{1}^{c}}x_{1}$ and $y_{2}=e_{L_{2}}z+e_{L_{2}^{c}}x_{2}$.
Since $y_{1}$, $y_{2}\in B$ we have 
\[
\alpha\leq f(y_{1})\Longrightarrow e_{L_{1}}\alpha\leq e_{L_{1}}f(y_{1})=e_{L_{1}}f(z)
\]
 and 
\[
\alpha\leq f(y_{2})\Longrightarrow e_{L_{2}}\alpha\leq e_{L_{2}}f(y_{2})=e_{L_{2}}f(z)
\]
 which gives $\alpha\leq f(z)$. 
\end{proof}
The following proposition is essential in the proof of the second
geometric form of the Hahn-Banach theorem (Thm.~\ref{thm:2nd=000020form}).
\begin{prop}
\label{prop:=000020strong=000020disj=000020compact=000020}Let $A$,
$C\subseteq\mathcal{G}_{E}$ be two non-empty internal sets, where
$A$ is a functionally compact set. \textcolor{black}{Assume that
the sets $A$ and $C$ are }strongly disjoint. Then there exists $m\in\mathbb{N}$
such that the sets $B:=[A_{\varepsilon}+B_{\rho_{\varepsilon}^{m}}(0)]$
and $D:=\langle C_{\varepsilon}+B_{\rho_{\varepsilon}^{m}}(0)\rangle$
are strongly disjoint. 
\end{prop}

\begin{proof}
First, since $A$ and $C$ are non-empty, there exist $a=[a_{\eps}]\in A$
and $c=[c_{\eps}]\in C$ such that 
\[
||a_{\eps}||_{\eps}+||c_{\eps}||_{\eps}\geq||a_{\eps}-c_{\eps}||_{\eps}\geq d(x,C_{\eps})\geq\inf_{x\in A_{\varepsilon}}d(x,C_{\varepsilon})
\]
which shows in particular that the net $(\inf_{x\in A_{\varepsilon}}d(x,C_{\varepsilon}))$
is $\rho$-moderate. We claim that 
\begin{equation}
\exists q\in\mathbb{N}:\,d(A,C)=\left[\inf_{x\in A_{\varepsilon}}d(x,C_{\varepsilon})\right]\geq\diff\rho^{q}.\label{eq:3.56}
\end{equation}
Indeed, for every $\varepsilon$, the set $A_{\varepsilon}$ is a
compact set of $E_{\eps}$ and the map $d(\cdot,C_{\varepsilon}):A_{\varepsilon}\longrightarrow\R$
is continuous. Thus, there exists a net $(x_{\varepsilon})$ of $\prod_{\eps\in I}A_{\eps}$
such that $\inf_{x\in A_{\varepsilon}}d(x,C_{\varepsilon})=d(x_{\varepsilon},C_{\varepsilon}).$
Since $A$ is sharply bounded, the net $(x_{\eps})\in\mathcal{M}_{E}$
and hence $x:=[x_{\varepsilon}]\in A$. If the net $(d(x_{\varepsilon},C_{\varepsilon}))$
is not invertible then $x\in_{L}C$ for some $L\subseteq_{0}I$ (see
Thm.~10, \emph{(i)} of \cite{GKV24}) which contradicts the fact
that $A$ and $C$ are strongly disjoint. Thus, \eqref{eq:3.56} follows.
One can easily show that \eqref{eq:3.56} gives 
\begin{equation}
\forall m\in\N_{>q}\,\forall^{0}\varepsilon:\,\left(A_{\varepsilon}+B_{\rho_{\varepsilon}^{m}}(0)\right)\cap\left(C_{\varepsilon}+B_{\rho_{\varepsilon}^{m}}(0)\right)=\emptyset.\label{eq:3.54}
\end{equation}
Indeed, if we assume that there exist $m>q$, a sequence $(\eps_{k})$
non-increasing and converging to $0$, and a sequence $(x_{\eps_{k}})$
such that 
\[
\forall k\in\N:\,x_{\eps_{k}}\in\left(A_{\eps_{k}}+B_{\rho_{\eps_{k}}^{m}}(0)\right)\cap\left(C_{\eps_{k}}+B_{\rho_{\eps_{k}}^{m}}(0)\right)
\]
then 
\[
\forall k\in\N:\,d(A_{\eps_{k}},C_{\eps_{k}})\leq d(x_{\eps_{k}},A_{\eps_{k}})+d(x_{\eps_{k}},C_{\eps_{k}})\leq2\rho_{\eps}^{m}
\]
which contradicts \eqref{eq:3.56}. Finally, \eqref{eq:3.54} proves
that the sets $B$ and $D$ are strongly disjoint (see the proof of
Prop.~\ref{prop:strong=000020disjoint=000020}). 
\end{proof}
\textcolor{black}{The following theorem is the second geometric form
of the Hahn-Banach theorem. }
\begin{thm}
\label{thm:2nd=000020form}Let $A$, $C\subseteq\mathcal{G}_{E}$
be two non-empty internal sets, with $A$ is a functionally compact
set. \textcolor{black}{Assume that the sets $A$ and $C$ are }strongly
disjoint, and have sharply bounded representatives consisting of convex
sets\textcolor{black}{. }Then, there exists a continuous $\rti$-linear
map $f=[f_{\eps}]\ra\rti$ non-identically null, $\alpha\in\rti$,
and $r\in\rti_{>0}$ such that \textcolor{black}{
\[
\forall x\in A\,\forall y\in C:\,f(x)+r\leq\alpha\leq f(y)-r.
\]
i.e.~the closed hyperplane $\{f=\alpha\}$ strictly separates $A$
and $C$. }
\end{thm}

For the proof of the theorem, we need the following
\begin{lem}
\label{lem:SupInf}Let $f:[f_{\eps}]:\mathcal{G}_{E}\ra\rti$ be a
continuous $\rti$-linear non-identically null. Then, the supremum
of the set $f(\overline{B}_{1}(0))$ and the infimum of the set $f(B_{1}(0))$
exist. 
\end{lem}

\begin{proof}
The existence of the supremum of $f(\overline{B}_{1}(0))$ follows
easily from the equality $f(\overline{B}_{1}(0))=[f_{\eps}(\overline{B}_{1\eps}(0))]$,
where $B_{1\eps}(0)$ is the unit ball of $E_{\eps}$. 

We prove now the existence of the infimum of $f(B_{1}(0))$. Since
$f_{\eps}$ is continuous and $(B_{1\eps})$ is sharply bounded, the
net $(f_{\eps}(B_{1\eps}))$ is sharply bounded as well. We can hence
set $\sigma=[\sigma_{\eps}]:=[\inf(f_{\eps}(B_{1\eps}(0))]$ which
is a lower bound of $f(B_{1}(0))$. Let now $q\in\N$. By definition
of $\sigma$, for a negligible net $(n_{\eps})$ there exists a net
$(z_{\eps})$ of $\prod_{\eps\in I}f_{\eps}(B_{1\eps}(0))$ such that
\[
\forall^{0}\eps:\,z_{\eps}\leq\sigma_{\eps}+n_{\eps}.
\]
The strongly internal set $\langle f_{\eps}(B_{1\eps}(0))\rangle$
satisfies the assumptions of Prop.~\ref{prop:ConvImlies}. Hence,
from \eqref{eq:1.2-1}, there exists $\overline{y}\in\langle f_{\eps}(B_{1\eps}(0))\rangle$
such that 
\[
\overline{y}\leq[z_{\eps}]+\diff\rho^{q+1}\leq\sigma+\diff\rho^{q+1}
\]
Let $(\overline{y}_{\eps})$ be a representative of $\overline{y}$.
Then, for all $\eps$ small, there exists $x_{\eps}\in B_{1\eps}(0)$
such that $\overline{y}_{\eps}=f_{\eps}(x_{\eps})$. Again, the strongly
internal set $B_{1}$ satisfies the assumptions of Prop.~\ref{prop:ConvImlies}.
Hence, \eqref{eq:1.2-1} implies that there exists $\overline{x}\in B$
such that $||x-\overline{x}||<\diff\rho^{q+1}|||f|||^{-1}$ where
$[x_{\eps}]=x$ and $|||f|||>0$ since $f$ is non-identically null.
It follows that 
\[
f(\overline{x})\leq\overline{y}+|||f||||x-\overline{x}||\leq\sigma+2\diff\rho^{q+1}\leq\sigma+\diff\rho^{q}.
\]
Therefore, $\sigma$ is the infimum of the set $f(B_{1}(0))$.
\end{proof}
\begin{proof}[Proof of Thm.~\ref{thm:2nd=000020form}]
\textcolor{black}{{} }By assumptions, there exists a representative
$(A_{\eps})$ of $A$ and a representative $(C_{\eps})$ of $C$ that
are sharply bounded and consisting of convex sets. By Lem.~\ref{lem:Sum}
and Prop.~\ref{prop:=000020strong=000020disj=000020compact=000020}
there exists $m\in\N$ such that the sets $B:=A+\overline{B}_{\diff\rho^{m}}(0)=[A_{\varepsilon}+\overline{B}_{\rho_{\varepsilon}^{m}}(0)]$
and $D:=C+B_{\diff\rho^{m}}(0)=\langle C_{\varepsilon}+B_{\rho_{\varepsilon}^{m}}(0)\rangle$
are strongly disjoint. Moreover, one can use Thm.~\ref{thm:geo=000020form=0000201}
and hence there exist a continuous $\rti$-linear map $f=[f_{\eps}]\ra\rti$
non-identically null and $\alpha\in\rti$ such that
\[
\forall x\in A\,\forall y\in C\,\forall\overline{z}\in\overline{B}_{1}(0)\,\forall z\in B_{1}(0):\,f(x+\diff\rho^{m}\overline{z})\leq\alpha\leq f(y+\diff\rho^{m}\text{\ensuremath{z}})
\]
which implies that 
\[
\forall x\in A\,\forall y\in C\,\forall\overline{z}\in\overline{B}_{1}(0)\,\forall z\in B_{1}(0):\,f(x)+\diff\rho^{m}f(\overline{z})\leq\alpha\leq f(y)+\diff\rho^{m}f(z).
\]
Setting $\overline{\sigma}=\sup_{z\in\overline{B}_{1}(0)}f(z)$, $\sigma=\inf_{z\in B_{1}(0)}f(z)$,
which exist thanks to Lem.~\ref{lem:SupInf}. Moreover, we clearly
have $\overline{\sigma}>0$ and $\sigma<0$ because $f$ is non-identically
null. It follows that 
\[
\forall x\in A\,\forall y\in C:\,f(x)+\diff\rho^{m}\overline{\sigma}\leq\alpha\leq f(y)+\diff\rho^{m}\sigma
\]
Setting $r=\frac{1}{2}\min(\diff\rho^{m}\overline{\sigma},-\diff\rho^{m}\sigma)\in\rti_{>0}$.
It follows that 
\[
\forall x\in A\,\forall y\in C:\,f(x)+r<\alpha<f(y)-r.
\]
Thus, the closed hyperplane $\tilde{H}=\{f=\alpha\}$ strictly separates
$A$ and $C$. 
\end{proof}

\end{document}